\documentclass[a4paper,11pt]{article} 

\usepackage{fullpage}         
\usepackage[english]{babel}
\usepackage[utf8]{inputenc}
\usepackage[style=alphabetic, backend=bibtex8, natbib]{biblatex}
\usepackage{authblk}          
\usepackage{amssymb}          
\usepackage{mathtools}	      
\usepackage{amsthm}	          
\usepackage{amsmath}
\usepackage{upgreek}          
\usepackage{mathrsfs}
\usepackage{graphicx}         
\usepackage{float}            
\usepackage{enumitem}
\usepackage{xcolor}           
\usepackage{hyperref}
\usepackage[title]{appendix}

\hypersetup{
    linktocpage = true,
    colorlinks  = true,                        
    linkcolor   = blue,
    citecolor   = blue,
    urlcolor    = blue
}

\numberwithin{equation}{section}

\DeclareMathOperator*{\tr}{tr}

\theoremstyle{definition}
\newtheorem{definition}{Definition}[section]

\theoremstyle{remark}

\theoremstyle{plain}
\newtheorem{assumption}{Assumption}
\newtheorem{lemma}[definition]{Lemma}

\newtheorem{corollary}[definition]{Corollary}
\newtheorem{theorem}[definition]{Theorem}
\newtheorem{proposition}[definition]{Proposition}

\begin{document}
\title{Higher-order spectral perturbation expansions II: \\ Kernel matrices and manifold learning}
\renewcommand\Affilfont{\itshape\small}

\author[1]{Bernhard Stankewitz}
\affil[1]{Institute of Mathematics, University of Potsdam, \normalfont \texttt{bernhard.stankewitz@uni-potsdam.de}}
\author[2]{Martin Wahl}
\affil[2]{Department of Mathematics, Bielefeld University, \normalfont \texttt{martin.wahl@math.uni-bielefeld.de}}
\date{\today}
\maketitle
\abstract{
  We study spectral concentration bounds for kernel matrices as approximation of the corresponding kernel integral operator.
  Results are established under weak assumptions on the data setting and the reproducing kernel relying only on a Mercer condition and a local Weyl law.
  This allows us to deal with key features of kernel matrices, such as large multiplicities, large effective dimension, and heavy-tailed distributions.
  Our results apply to infinite dimensional principal component analysis, manifold learning, and Bayesian nonparametric statistics.
  We illustrate this via two prototypical examples: The heat kernel on the sphere and a wavelet prior from Bayesian nonparametrics.\smallskip

}

\section{Introduction}
\label{sec_Introduction}

Kernel matrices and their spectral properties play a fundamental role in statistics and machine learning.
For instance, spectral regularization algorithms \cite{SmaleZhou2007LearningTheoryEstimates,MR2600634} selectively remove noise-amplifying directions.
In non-parametric and high-dimensional regression \cite{MR2766864,MR4209491,MR5005628} as well as in Bayesian statistics \cite{MR2459226}, their eigenvalue decay governs convergence rates and characterizes high-dimensional phenomena. The fact that the eigenvectors of kernel matrices encode the underlying geometry of data is heavily used in classification and spectral clustering \cite{10.5555/345662,MR2396807}, dimension reduction \cite{6790375,MR1966051,MR2177937}, and manifold learning \cite{MR4452681}. 
Besides that, via the neural tangent kernel, kernel matrices also play a role in explaining the statistical behavior of (wide) neural networks \cite{MR4255118}. 

In many applications, spectral concentration bounds that relate kernel matrices and empirical covariance operators to their population counterparts are central tools of the analysis. Starting with the work by Giné and Koltchinskii \cite{MR1781185,MR1652327}, eigenvalues and eigenvectors of kernel matrices have been intensively studied in the literature. More recent work focuses on relative eigenvalue error bounds (see e.g.~\cite{MR2274441,pmlr-v99-ostrovskii19a,BarzilaiShamir}) and the predictive performance of the eigenvalues of kernel matrices. In the related direction of Laplacian Eigenmaps and Diffusion Maps, eigenvalues and eigenvectors of graph Laplacians as spectral approximation of Laplace-Beltrami operator have attracted significant recent interest \cite{MR4130541,MR4393800,MR4548608,MR4452681,TLV25,CD25}. 

This paper is the second step in an attempt to derive sharp concentration results for eigenvalues and eigenprojectors based on higher-order spectral perturbation expansions.
It replaces the idealized, prototypical setting of sub-Gaussian distributions and simple eigenvalues from \cite{Wahl2026PertubationExpansionsI} to more practical and realistic assumptions.

The first motivating example for this work is the kernel matrix problem in the context of manifold learning. 
Assume that we observe $n$ independent random variables uniformly distributed on the $d$-dimensional sphere (as a prototype of a compact submanifold in high dimensions) and construct the kernel matrix using the heat kernel on the sphere at time $t$. 
This kernel matrix approximates the heat semigroup at time $t$, that is the integral operator associated with the heat kernel. 
Under which conditions on the three involved parameters $t$, $d$ and $n$ do we have spectral convergence? 
The generator of the heat semigroup is the Laplace-Beltrami operator on the sphere. 
Similarly, we can take the difference quotient of the kernel matrix with small $t$. 
When are the associated eigenvalues and eigenvectors approximations of the eigenvalues and eigenfunctions of the Laplace-Beltrami operator on the sphere? 
These questions lie at the basis of non-linear dimensionality reduction methods, such as Laplacian Eigenmaps and Diffusion Maps. 
They are also difficult to answer for several reasons. 
First, there are large multiplicities for small eigenvalues. 
Second, there is a large effective dimension, especially for small $t$. 
Third, the Karhunen-Loève coefficients are usually neither independent nor uniformly bounded.

A second example stems from Bayesian nonparametric statistics.
In the context of Gaussian process (GP) regression, when the estimation target function is endowed with a GP prior, contraction rates of the posterior are characterized by the eigenvalue decay of the population kernel operator.
In modern setting with large sample sizes \( n \ge 10^6 \), however, computing the full posterior is computationally infeasible due to its cubic computational cost and statistical guarantees for approximations, e.g., via variational Bayes \cite{NiemanEtal2022ContractionRates, Nieman2023UncertaintyQuantification, NiemanSzabo2025AdaptiveVB}, require some spectral concentration guarantees for the kernel matrix of the prior.
These requirements become more severe when our guarantees should cover fully numerical approximations \cite{WangEtal2019ExactGPs, WengerEtal2022ComputationalUncertainty, GardnerEtal2018GPyTorch} that rely on conjugate-gradient solvers or Lanczos approximations of the singular value decomposition of the kernel matrix that remain implicit in the variational theory.
So far, here, guarantees are only available under a singular eigenvalue regime of the kernel operator together with uniform (moment) bound on the eigenfunctions \cite{StankewitzSzabo24}.
Motivated by this, we consider a concentration guarantees for a truncated wavelet prior, as it is standard in the Bayesian literature \cite{GhosalvdVaart2017FundamentalsOfBayes}, that defies of these restrictions having both very large eigenvalue multiplicities and no uniform bounds on its eigenfunctions.

The paper is organized as follows.
Section \ref{sec_SpectralConcentrationForKernelMatrices} introduces the data setting in detail and presents the main results for kernel matrices.
Section \ref{sec_Examples} applies the main results to our two motivating examples.
Section \ref{ssec_ProofsOfTheProbabilisticEstimates} gives the proof of the probabilistic estimates.
Sections \ref{sec_HigherOrderSpectralPerturbationBounds} and \ref{sec_ProofsOfThePerturbationBounds} develops the general higher-order spectral perturbations bounds that are at the core of our main results and of independent interest.
Section \ref{sec_AddtionalProofsOfAuxiliaryResults} collects additional proofs.


\section{Spectral concentration for kernel matrices}
\label{sec_SpectralConcentrationForKernelMatrices}

\subsection{Problem setup}
\label{ssec_ProblemSetup}

We consider the following scenario. Let \( K \in \mathbb{R}^{n \times n} \) be a normalized kernel matrix defined by
\begin{align}
  K_{i j}: = n^{-1} k( X_{i}, X_{j} ),
\end{align}
where $ X_1, \dots, X_n $ are i.i.d.~random variables and $ k: \mathcal{X} \times \mathcal{X} \rightarrow \mathbb{R}$ is a function.
In the sequel, we will work under the following assumptions.

\begin{assumption}[Data setting]
  \label{ass_data}
  \( X_{1}, \dots X_{n} \) are i.i.d.~copies of a random variable \( X \) taking values in a polish space $(\mathcal{X}, d)$ equipped with its Borel-$ \sigma $-algebra $ \mathscr{B}_{\mathcal{X}}$. We denote by $ \mathbb{P}^X$ the distribution of $X$.
\end{assumption}

\begin{assumption}[Kernel function]
  \label{ass_Kernel}
  \( k: \mathcal{X} \times \mathcal{X} \to \mathbb{R} \) is symmetric and positive semidefinite, meaning that \( k(x, y) = k(y, x) \) for all \( x, y \in \mathcal{X} \) and \( \sum_{i, j = 1}^{n} a_{i} a_{j} k(x_{i}, x_{j}) \ge 0 \) for all \( n \in \mathbb{N} \), all \( a_1,\dots,a_n \in \mathbb{R} \), and all \( x_1,\dots,x_n \in\mathcal{X} \).
\end{assumption}

\begin{assumption}[Mercer's condition]
  \label{ass_Mercer}
  $ k: \mathcal{X} \times \mathcal{X} \to \mathbb{R} $ is continuous and
  \begin{align}
  \label{Eq_Mercer}
    \int_{\mathcal{X}} k(x, x) \, \mathbb{P}^X(dx) < \infty.
  \end{align}
\end{assumption}

Under Assumption \ref{ass_Kernel}, the kernel matrix $K$ is symmetric and positive semi-definite. 
Hence, by the spectral theorem, there are eigenvalues $\widehat\lambda_1\geq \dots\geq \widehat\lambda_n$ and an associated orthonormal basis of eigenvectors $\widehat v_1,\dots,\widehat v_n $ in $\mathbb{R}^n$ with respect to the Euclidean norm on \(\mathbb{R}^n\) such that 
\begin{align}
    K\widehat v_j=\widehat\lambda_j\widehat v_j\qquad \text{for all } 1\leq j\leq n,
\end{align}
i.e.,  \( K \) has the spectral representation
\begin{align}
  K = \sum_{j = 1}^{n} \widehat{\lambda}_{j} \widehat{v}_{j} \widehat{v}_{j}^{\top}.
\end{align}

Under Assumption \ref{ass_Mercer}, then, the following holds. 
First, continuity of $k$ implies that $k$ is measurable with respect to the Borel $\sigma$-algebra $\mathscr{B}_{ \mathcal{X} \times \mathcal{X}} $ and the latter is equal to the product-$\sigma$-algebra $\mathscr{B}_X \otimes \mathscr{B}_{\mathcal{X}} $, see \cite[Proposition 4.1.7]{Dudley2002RealAnalysisAndProbability}.
Using also Equation \eqref{Eq_Mercer}, we arrive at $k \in L^2(\mathbb{P}^X \otimes \mathbb{P}^X)$, as can be seen from the inequality $ k(x,y)^2 \leq k(x,x)\cdot k(y,y) $, which follows from the positive semidefiniteness of \( k \), Fubini’s theorem, and \eqref{Eq_Mercer}.
Thus, we we can consider the kernel integral operator
\begin{align}
  L_k: L^{2}(\mathbb{P}^X) \to     L^{2}(\mathbb{P}^X), \qquad 
                         f \mapsto \int_{\mathcal{X}} k(\cdot, y) f(y) \, \mathbb{P}^X(dy).
\end{align}

Using also Assumption \ref{ass_Kernel}, $ L_k $ is and compact linear operator, see \cite[Theorem 4.2.16]{D07}.
It is self-adjoint and our assumptions are also sufficient for it to be positive, see Lemma \ref{lem_LKIsPositive}.
Hence, by the spectral theorem for compact self-adjoint operators, there is a (possibly finite) sequence of positive eigenvalues $\lambda_1 \ge \lambda_2 \ge \dots >0$ and a (in this case also finite) orthonormal system of eigenfunctions $ \phi_1,\phi_2, \dots $ in $L^2(\mathbb{P}^X)$ such that $(\phi_j)$ is an orthonormal basis of $(\operatorname{ker} L_k)^\perp$ satisfying
\begin{align}
  L_{k} \phi_j = \lambda_j \phi_j, \quad \text{for all } j \ge 1, 
  \qquad 
  \operatorname{tr}(L_{k}) = \sum_{j \ge 1} \lambda_j 
                           = \int_\mathcal{X} k(x,x) \, \mathbb{P}^X(dx) 
                           < \infty
\end{align}
and
\begin{align}
  k(x,y) 
  =
  \sum_{j \ge 1} \langle k(x, \cdot), \phi_{j} \rangle_{L^{2}} \phi_{j}(y)
  = 
  \sum_{j \ge 1} \lambda_j \phi_j(x) \phi_j(y)
\end{align}
for $ \mathbb{P}^X \otimes \mathbb{P}^X $-almost all $ x, y \in \mathcal{X} $, where the right-hand side convergences in $L^2(\mathbb{P}^X \otimes \mathbb{P}^X) $.
The latter latter series also converges absolutely and uniform on compact sets, which is the content of Mercer theorem.
\begin{lemma}[Mercer’s theorem]
  \label{lem_Mercer}
  Under the above Assumptions \ref{ass_data}, \ref{ass_Kernel}, and \ref{ass_Mercer}, the eigenfunctions \( \phi_{j} \) are continuous and we have that
  \begin{align}\label{eq_Mercer}
    k(x,y) = \sum_{j \ge 1} \lambda_j \phi_j(x) \phi_j(y)
  \end{align}
  for all $x,y\in\mathcal{X}$, where the right-hand side converges absolutely and uniformly on compacts.
\end{lemma}

\noindent For a proof in the special case of $\mathcal{X}$ being a bounded domain of $\mathbb{R}^p$ equipped with the Lebesgue measure see \cite[Theorem 4.23]{B20}. Note that the same argument also works under our more general setup replacing the bounded domain assumption with the fact that \( \mathbb{P}^{X} \) is finite. For an alternative proof in the special case of self-adjoint semigroups under a trace class hypotheses see \cite[Theorem 7.2.5]{D07}. The latter setting is considered in the heat kernel example in Section \ref{ssec_TheSphere}.

In addition to Assumptions \ref{ass_data}--\ref{ass_Mercer}, we assume that the eigenfunctions of the kernel operator satisfy a local Weyl law, i.e., that their squares are bounded on average in each block. Importantly, we account for potential multiplicities of the population eigenvalues \( (\lambda_{j})_{ j \ge 1} \).
Hence, let \( (\mu_{r})_{r \ge 1} \) denote the sequence of distinct eigenvalues of \( \Sigma \), \( I_{r} := \{ j \ge 1: \lambda_{j} = \mu_{r} \} \) and \( m_{r} = | I_{r} | \) the multiplicity of \( \mu_{r} \).  

\begin{assumption}[Local Weyl law]
  \label{ass_LocalWeylLaw} 
  There exists a constant \( L \ge 1 \) such that the  orthonormal basis \( ( \phi_{j} )_{j \ge 1} \) satisfies that almost surely
  \begin{align}
    \sum_{j\in I_r}\phi_j^2(x)\leq L m_r
    \qquad \text{ for all } r \geq 1\text{ and all } x\in\mathcal{X}.
  \end{align}
\end{assumption}

\noindent Our theory is also able to cover weaker assumptions that are frequently satisfied by the \( (\phi_{j})_{j \ge 1} \), such as a bound of the form
\begin{align}
      \sum_{j =1}^{m} \phi_{j}^{2}(x) \le L m^{1+\gamma}\qquad \text{ for all } J \geq 1\text{ and all } x\in\mathcal{X}
\end{align}
for some \( \gamma \geq 0 \).
This is true, for example, for the important class of Matèrn kernels, see \cite{Rosa2025L2ContractionRates} and the case \( \tau = 1 \) was already partially explored by \cite{Wahl2024LaplacianEigenmaps} in the context of manifolds learning. 
Such bounds also feature prominently in other nonparametric problems, see \cite{Massart2007ModelSelection, BarronEtal1999ModelSelection}. 
The stronger Assumption~\ref{ass_LocalWeylLaw}, however, leads to the cleanest probabilistic estimates, so that we do not treat weaker conditions in this paper.

A fundamental problem is to relate the spectral characteristics of \( K \) to those of the integral kernel operator $ L_k$. 
In many of the applications mentioned in the beginning of the introduction, key features of $L_k$ are spectral decay, large multiplicities, large traces, and eigenfunctions that are not uniformly bounded. 
This provides an interesting random operator setting from high-dimensional statistics and machine learning, which is not yet as well understood as the spiked covariance model or covariance matrices under the Marchenko--Pastur limit.


\subsection{The kernel trick}
\label{ssec_TheKernelTrick}

The eigenvalue and eigenvector problem for the kernel matrix can be reformulated in terms of an empirical covariance operator. First, introduce the Hilbert space
\begin{align}\label{eq:RKHS}
    \mathcal{H}
  = 
    \Big\{ u=\sum_{j \geq 1}\beta_{j} \phi_{j}: 
           \sum_{j \geq  1} \frac{ \beta_{j}^{2} }{ \lambda_{j} } 
         < \infty
    \Big\}
  \quad \text{ with } \quad 
  \langle u, v \rangle_{ \mathcal{H} } 
  = 
  \sum_{j \geq 1}
  \frac{ \langle u, \phi_{j} \rangle_{ L^{2} } 
         \langle v, \phi_{j} \rangle_{ L^{2} }
       }
       { \lambda_{j} }.
\end{align}
Then $\mathcal{H}$ is a separable Hilbert space with orthonormal basis $(u_j)_{j\geq 1}$ given by $u_j=\lambda_j^{1/2}\phi_j$ for all $j\geq 1.$ 
Moreover, using Lemma \ref{lem_Mercer}, it is easy to check that $ k(x, \cdot) \in \mathcal{H} $ for all $ x \in \mathcal{X} $ and that the reproducing property $ h(x)=\langle h, k(x,\cdot) \rangle_{\mathcal{H}} $ holds for all $ h \in \mathcal{H} $ and all $ x \in \mathcal{X} $.
Consequently, \( \mathcal{H} \) is the unique reproducing kernel Hilbert space induced by \( k \), see also \cite{SteinwartChristmann2008SVMs, Wainwright2019HDStats}.
Moreover, the continuity of \( k \) and the reproducing property imply that the map $ x \mapsto k(x,\cdot)$  from $(\mathcal{X},d)$ to $(\mathcal{H},\|\cdot\|_\mathcal{H})$ is continuous such that $ k(X, \cdot) $ is a random variable taking values in $ \mathcal{H} $.
Since 
\[
  \mathbb{E}\|k(X,\cdot)\|_{\mathcal{H}_k}^2
  =
  \int k(x, x) \mathbb{P}^{X}(d x) < \infty,
\]
it is strongly square-integrable and we can introduce the covariance operator
\begin{align}
  \Sigma = \mathbb{E}(k(X,\cdot) \otimes k(X,\cdot)),
\end{align}
which is a positive, self-adjoint trace class operator (ee e.g.~\cite[Theorem 7.2.5]{HE15}). The following close relationship of the covariance operator $\Sigma$ and integral operator $L_k$ is well-known (see e.g.~\cite[Theorem 4.1]{MR2177937}).

\begin{lemma}
  \label{lem_L_Sigma}
  $\Sigma$ is the restriction of $L_k$ to $\mathcal{H}$. In particular, the spectral decomposition of $\Sigma$ is given by
  \begin{align}
    \Sigma=\sum_{j\geq 1}\lambda_j u_j\otimes u_j,
  \end{align}
  where $\lambda_1\geq \lambda_2\geq \dots>0$ are the non-zero eigenvalues of $L$ and $u_j=\lambda_j^{1/2}\phi_j$ for all $j\geq 1$.
\end{lemma}

\begin{proof}
  For the first statement note that for $f\in\mathcal{H}$ we have
  \begin{align}
      \Sigma f 
    = \mathbb{E} (k(X,\cdot) \langle f, k(X, \cdot) \rangle_{\mathcal{H}}) 
    = \mathbb{E} (k(X,\cdot)f(X))
    = L f,
  \end{align}
  where we used the reproducing property in the second equality.
  The second statement is a consequence of the first one using the definition of $\mathcal{H}$.
\end{proof}

We summarize the above results in the following lemma.

\begin{lemma}
\label{lem:KL}
  Assume that Assumptions \ref{ass_data}, \ref{ass_Kernel}, and \ref{ass_Mercer} hold. Then the $\mathcal{H}$-valued random variable $k(X,\cdot)$ is strongly square-integrable with (unnormalized) covariance operator $\Sigma$. Moreover, defining the Karhunen-Loève coefficients 
  \begin{align}
        \eta_j=\frac{\langle k(X,\cdot),u_j\rangle_{\mathcal{H}}}{\sqrt{\lambda_j}}=\phi_j(X),\qquad j\geq 1,
    \end{align}
    we have the series representation
  \begin{align}
        k(X,\cdot)=\sum_{j\geq1}\sqrt{\lambda_j}\eta_ju_j\qquad \mathbb{P}^X\text{-a.s.},
    \end{align}
    with convergence with respect to the norm $\sqrt{\mathbb{E}\|\cdot\|^2_\mathcal{H}}$.
\end{lemma}

In the data setting, we can introduce the embedded random variables 
\begin{align}
k(X_1, \cdot), \dots, k(X_n, \cdot)    
\end{align} 
taking values in $ \mathcal{H} $, and the empirical covariance operator 
\begin{align}
  \widehat\Sigma=\frac{1}{n}\sum_{i=1}^nk(X_i,\cdot)\otimes k(X_i,\cdot),
\end{align}
which is a positive, self-adjoint linear operator on \( \mathcal{H} \) of rank at most $ n $.
Similar as above the empirical covariance operator $\widehat\Sigma$ has the same non-zero eigenvalues as $K$ and the eigenvectors can be easily translated using the sampling operator $ S_{n}: \mathcal{H} \to \mathbb{R}^n $ given by $ S_nh = (h(X_1), \dots, h(X_n))^\top$.

\begin{lemma}
\label{lem_L_Sigma_empirical}
  The spectral decomposition of $ \widehat{\Sigma} $ is given by
  \[
    \widehat\Sigma=\sum_{j=1}^n\widehat\lambda_j \widehat u_j\otimes \widehat u_j
  \]
  with $ \widehat \lambda_1 \ge \dots \ge \widehat \lambda_n $ being the eigenvalues of $K$ and an orthonormal system $ \widehat u_1, \dots, \widehat u_n $ in $ \mathcal{H} $ satisfying $(\widehat u_j(X_1), \dots, \widehat u_j(X_n))^\top = (n \widehat{\lambda}_j)^{1/2} \widehat v_j $ for all $ j = 1, \dots, n $ such that $ \widehat \lambda_j >0 $.
\end{lemma}

It follows from Lemmas \ref{lem_L_Sigma} and \ref{lem_L_Sigma_empirical} that the eigenvalue and eigenvector problem for $K$ can be translated to the eigenvalue and eigenvector problem of the (empirical) covariance operators $\widehat\Sigma$ and  $\Sigma$.  This fact is often called the kernel trick, as one only has to compute the inner products $k(X_i,X_j)$, while the explicit coordinates of the embedded random variables \( k(X_1, \cdot), \dots, k(X_n, \cdot) \) in the typically infinite-dimensional reproducing kernel Hilbert space $\mathcal{H}$ are not required.


\subsection{Main result}
\label{ssec_MainResults}

In this section, we formulate our main result in the kernel matrix setting. The behavior of our bounds, which will be formulated for the \( r \)-th eigenblock, is governed by the relative rank.

\begin{definition}[Relative rank]
    For $r\geq 1 $ the relative rank is defined by
    \begin{align}
  \mathbf{r}_{r}(\Sigma): = \sum_{s \ne r} \frac{m_{s} \mu_{s}}{| \mu_{s} - \mu_{r} |} 
                          + \frac{m_{r} \mu_{r}}{g_{r}},
\end{align}
where \( g_{r}: = \min(\mu_{r - 1} - \mu_{r}, \mu_{r} - \mu_{r + 1}) \) is the \( r \)-th spectral gap.
\end{definition}

The following result analyzes the eigenvalues $(\widehat{\lambda}_j)_{j\in I_r}$ and eigenvectors $(\widehat v_j)_{j \in I_r}$ of the kernel matrix $K$ within a block $I_r$. It compares these to the eigenvalue $\mu_r$ and the eigenfunctions $(\phi_j)_{j\in I_r}$ of the integral kernel operator $L_k$, where the latter are evaluated at the observations and considered up to an orthogonal transformation.

\begin{theorem}[Bounds for eigenvalues and eigenspaces]
  \label{thm_BoundsForEigenvaluesAndEigenspaces}
  Assume that Assumptions \ref{ass_data}--\ref{ass_LocalWeylLaw} hold and let $ \gamma \ge 1 $.
  Then there are constants $c, C>0$ depending only $ \gamma $ and $ L $ such that the following holds.
  Let $r \ge 1$ be such that 
  \begin{align}
    \frac{\mu_{r}}{g_{r}} \mathbf{r}_r(\Sigma) \leq c\frac{n}{\log n}.
  \end{align}
  Then the eigenvalues and eigenvectors of the kernel matrix satisfy 
  \begin{align}\label{eq:eva:bound:KS}
    \frac{1}{\sqrt{m_r}} \Big\| \Big(\frac{\widehat{\lambda}_j-\mu_r}{\mu_r}\Big)_{j \in I_r} \Big\|_2
    \le
    C \Big( \sqrt{\frac{m_r \log n}{n}} + \mathbf{r}_{r}(\Sigma) \frac{\log n}{n} \Big)
  \end{align}
  and
  \begin{align}\label{eq:ev:bound:KS}
    & \ \ \ \
    \inf_{O \in O(m_r)} 
    \frac{1}{ \sqrt{m_r} } \Bigg\| (\widehat v_j)_{j \in I_r}O - \frac{1}{\sqrt{n}}
                                                                \begin{pmatrix}
                                                                  \phi_j(X_1) \\
                                                                  \vdots      \\
                                                                  \phi_j(X_n) 
                                                                \end{pmatrix}_{j \in I_r} 
                           \Bigg\|_2
    \\
    &\le 
    C \Big( \sqrt{ \sum_{s \ne r} \frac{m_s \mu_r \mu_s}{(\mu_r - \mu_s)^2} \frac{\log n}{n} } 
          + \sqrt{ \mathbf{r}_{r}(\Sigma) \frac{\log n}{n}} 
          + \frac{\mu_r}{g_{r}} \mathbf{r}_{r}(\Sigma) \frac{\log n}{n}
      \Big)    
    \notag
  \end{align}
  with probability at least \( 1 - n^{-\gamma} \).
\end{theorem}

\noindent A proof of Theorem \ref{thm_BoundsForEigenvaluesAndEigenspaces} can be found in Section \ref{sec_AddtionalProofsOfAuxiliaryResults}.
If we combine the upper bounds into one (possibly less sharp) term, then \eqref{eq:eva:bound:KS} and \eqref{eq:ev:bound:KS} imply 
\begin{align}
    \frac{1}{\sqrt{m_r}} \big\| (\widehat{\lambda}_j-\mu_r)_{j \in I_r} \big\|_2
    \le
    C g_r\sqrt{ \frac{\mu_r}{g_{r}}\mathbf{r}_{r}(\Sigma) \frac{\log n}{n} }
  \end{align}
  and
  \begin{align}
    & \ \ \ \
    \inf_{O \in O(m_r)} 
    \frac{1}{ \sqrt{m_r} } \Bigg\| (\widehat v_j)_{j \in I_r}O - \frac{1}{\sqrt{n}}
                                                                \begin{pmatrix}
                                                                  \phi_j(X_1) \\
                                                                  \vdots      \\
                                                                  \phi_j(X_n) 
                                                                \end{pmatrix}_{j \in I_r} 
                           \Bigg\|_2
    \le 
    C \sqrt{ \frac{\mu_r}{g_{r}} \mathbf{r}_{r}(\Sigma) \frac{\log n}{n} }
  \end{align}
  with probability at least \( 1 - n^{-\gamma} \).

Eigenvalues of kernel matrices have been extensively studied in the literature. More recent work focuses on relative eigenvalue error bounds (see e.g.~\cite{MR2274441,pmlr-v99-ostrovskii19a,BarzilaiShamir}). Such eigenvalue concentration bounds up to constants are used in prediction problems and lead to effective rank conditions (see \cite{MR4263288,MR4652993}). However, such bounds do not include joint eigenvalue and eigenprojection concentration bounds, as it is often needed in dimensionality reduction methods (functional and kernel PCA, Diffusion Maps, Laplacian Eigenmaps, etc.).

Similarly as in the previous work \cite{JirakWahl2018PerturbationBoundsUnderRelativeGap,JirakWahl2023RelativePerturbationBounds,Wahl2026PertubationExpansionsI}, the bounds are formulated in terms of the relative rank. The difference is that we explicitly work with kernel matrices under realistic assumptions. This leads to challenges in terms of large multiplicities $m_r$ and eigenfunctions $\phi_j$ that are neither uniformly bounded (as assumed in \cite{JirakWahl2018PerturbationBoundsUnderRelativeGap,JirakWahl2023RelativePerturbationBounds}), nor lead to independent and sub-Gaussian Karhunen-Loève coefficients (as assumed in \cite{Wahl2026PertubationExpansionsI}).

While a linear analog of \eqref{eq:ev:bound:KS} is well-known in the context of principal component analysis and the left-hand side of \eqref{eq:ev:bound:KS} is part of the different formulations of the Davis-Kahan bound, it has, to the best of our knowledge, not yet appeared in the literature on kernel matrices. It shows the non-linear dimensionality reduction property of kernel matrices. The eigenvectors of kernel matrices give a new coordinate system for the embedded data points, approximating the eigenfunctions of the kernel integral operator evaluated at the data points.

We conclude this section by presenting two  main steps in the proof of the main result. To formulate the first step, let 
\begin{align}
    P_{r}: = \sum_{j \in I_{r}}  u_{j} \otimes u_{j}\qquad \text{and}\qquad \widehat{P}_{r}: = \sum_{j \in I_{r}} \widehat{u}_{j} \otimes \widehat{u}_{j}
\end{align}
be the projections on the \( r \)-th eigenspace of \( \Sigma \) and \( \widehat{\Sigma} \) respectively, 
\begin{align}
  R_{r}: & = \sum_{s \ne r} (\mu_{s} - \mu_{r})^{-1} P_{s}
\end{align}
be the $r$-th reduced resolvent, and $ E: = \widehat{\Sigma} - \Sigma$.

\begin{proposition}[Deterministic bounds]
  \label{prp_DeterministicBounds}
   Suppose that
   \begin{align}
                   \delta_r^\prime:= \max\big(\| |R_{r}|^{1/2} E |R_{r}|^{1/2} \|_{\infty}, g_{r}^{-1/2} \| |R_{r}|^{1/2} E  P_{r}        \|_{\infty}
                , g_{r}^{-1} \| P_{r} E P_{r} \|_{\infty}\big)< 1 / 4 - \varepsilon , 
  \end{align}
  for some $\epsilon\in(0,1/4)$. Then there is a constant $C>0$ depending only on $\varepsilon$ such that 
  \begin{align}
          \|(\widehat\lambda_j-\mu_r)_{j\in I_r}\|_{2}
    \leq
          C ( \| P_{r} E P_{r} \|_{2} + g_r^{1/2}\| | R_{r} |^{1/2} E P_{r} \|_{2}\delta_r^\prime)
  \end{align}
  and
  \begin{align}
    \inf_{O \in O(m_r)} 
    \Bigg\| (\widehat v_j)_{j \in I_r} O&- \frac{1}{\sqrt{n}}  \begin{pmatrix}
                                                          \phi_j(X_1) \\
                                                          \vdots      \\
                                                          \phi_j(X_n) 
                                                        \end{pmatrix}_{j \in I_r}  
    \Bigg\|_2
    \leq 
     \Big\| \Big(\frac{\widehat{\lambda}_{j}}{\mu_r} - 1\Big)_{j \in I_{r}} \Big\|_{2}\\
    &+C(\|  R_{r} E P_{r} \|_{2}+\mu_r^{-1/2}\| |R_{r}|^{1/2} E P_{r} \|_{2}+g_r^{-1/2}\| | R_{r} |^{1/2} E P_{r} \|_{2}\delta_r^\prime)
    \notag
  \end{align}
  where the infimum is taken over all orthogonal matrices of dimension $|I_r|$.
\end{proposition}

Proposition \ref{prp_DeterministicBounds} is a consequence of the higher-order spectral perturbation expansions that we present in Section \ref{sec_HigherOrderSpectralPerturbationBounds}. The operator norms in $\delta_r^\prime$ is bounded in Section \ref{ssec_ProofsOfTheProbabilisticEstimates} using a standard operator Bernstein inequality for sums of independent random self-adjoint operator. The following proposition gives a (surprisingly) simple probabilistic description of theinvolved perturbation terms. It relies on the local Weyl law in Assumption \ref{ass_LocalWeylLaw}. Recall that a random variable $Z$ is called sub-exponential with parameter $K$, if 
\begin{align}
  \|Z\|_{\psi_1}:=\sup_{p\geq 1}p^{-1}\|Z\|_{L^p}\leq  K. 
\end{align}
This is equivalent to 
\begin{align}
 \mathbb{P}(|Z|\geq Ct)\leq 2\exp(-t/K)\qquad  \text{for all }t\geq 0,
\end{align}
where $C>0$ is an absolute constant.

\begin{proposition}[Probabilistic bounds]
  \label{prp_BoundsOnTheLinearAndQuadraticTerm}
  If Assumptions \ref{ass_data}--\ref{ass_LocalWeylLaw} are satisfied, then the following holds.
\begin{itemize}
    \item[(i)] The random variable $\| P_{r} E P_{r} \|_{2}^2 $ is sub-exponential with parameter
  \begin{align}
    16\sqrt{2}L^2\frac{m_r^2\mu_r^2}{n}.
  \end{align}
  \item[(ii)] The random variable $| | R_{r} |^{1/2} E P_{r} \|_{2}^2 $ is sub-exponential with parameter
  \begin{align}
    8\sqrt{2}L^2\frac{m_r\mu_r}{n}\sum_{s\neq r}\frac{m_s\mu_s}{|\mu_r-\mu_s|}.
  \end{align}
  \item[(iii)] The random variable $\| R_{r}  E P_{r} \|_{2}^2 $ is sub-exponential with parameter
  \begin{align}
    8\sqrt{2}L^2\frac{m_r\mu_r}{n}\sum_{s\neq r}\frac{m_s\mu_s}{(\mu_r-\mu_s)^2}.
  \end{align}
\end{itemize} 
\end{proposition}



\section{Two examples}
\label{sec_Examples}

\subsection{Example 1: The sphere and manifold learning}
\label{ssec_TheSphere}

For \( d \ge 1 \), consider
\begin{align}
  \mathcal{X} = \mathbb{S}^{d}(R): = \{ x \in \mathbb{R}^{d + 1}: \| x \|_2 = R \} 
\end{align}
with $\|\cdot\|_2$ being the Euclidean norm in $\mathbb{R}^{d+1}$ and radius \( R > 0 \) chosen such that the volume of $ \mathbb{S}^d(R) $ is equal to one. From the volume formula of the \( (d + 1) \)-dimensional Euclidean ball with radius \( R > 0 \) and Stirling's approximation of the Gamma function, it follows that
\begin{align}
  R = \Big(\frac{\Gamma(\frac{d+1}{2})}{2 \pi^{\frac{d+1}{2}}}\Big)^{\frac{1}{d}}
  \qquad \text{and} \qquad
  c\sqrt{d}\leq  R\leq C\sqrt{d}
\end{align}
for some absolute constants $ c, C > 0 $.
We let $ \mathbb{P}^X $ be the uniform distribution on $ \mathbb{S}^d(R) $, that is the volume measure on $ \mathbb{S}^d(R) $, so that Assumption \ref{ass_data} is satisfied. For $t>0$, let $ k_t(x, y) $ be the heat kernel on $ \mathbb{S}^d(R) $, which satisfies Assumptions \ref{ass_Kernel} and \ref{ass_Mercer} by \cite[Theorem 7.13 and Theorem 7.20]{MR2569498}. The associated integral kernel operator  is the heat semigroup on \( \mathbb{S}^{d}(R) \) at time $t$, and we also write $e^{- t \Delta_{\mathbb{S}^{d}(R)}}$ instead of $L_{k_t}$, i.e. 
\begin{align}
  e^{- t \Delta_{\mathbb{S}^{d}(R)}} f(x)=L_{k_t}f(x) = \int_{\mathbb{S}^d(R)} f(y)  k_{t}(x, y) \mathbb{P}^{X}(d y) 
  \qquad f \in L^{2}(\mathbb{S}^{d}(R)).
\end{align}
Moreover, there are real numbers $0=\nu_1<\nu_2<\nu_3<\cdots \rightarrow\infty$, an orthonormal basis \( (Y_{r,j})_{r \ge 1,1\leq j\leq m_r} \) of $ L^{2}(\mathbb{S}^{d}(R))$ consisting of smooth eigenfunctions of \( \Delta_{\mathbb{S}^{d}(R)}\), such that (see e.g.~\cite[Chapter 10.4]{MR2569498})
\begin{align}
  e^{- t \Delta_{\mathbb{S}^{d}(R)}}Y_{r,j}=e^{-t\nu_r}Y_{r,j},\qquad \Delta_{\mathbb{S}^{d}(R)}Y_{r,j}=\nu_rY_{r,j}\quad \text{ for all $1\leq j\leq m_r$ and $r\geq 1$,}
\end{align}
and the heat kernel \( k_{t} \) admits the representation 
\begin{align}
  k_{t}(x, y) = \sum_{r = 1}^{\infty} e^{-t\nu_r} \sum_{j=1}^{m_r}Y_{r,j}(x) Y_{r,j}(y), 
  \qquad x, y \in \mathbb{S}^{d}(R)
\end{align}
with  uniform convergence in $x,y\in \mathbb{S}^{d}(R)$. The distinct eigenvalues and their multiplicities are explicit.

\begin{lemma}[Eigenvalues and their multiplicities for the sphere]
  \label{lem_HeatKernelEigenvaluesAndTheirMultiplicities}
  The distinct eigenvalues of the heat kernel operator on \( \mathbb{S}^{d}(R) \) and their multiplicities are given by
  \begin{align}
    \mu_{r, t} = e^{- t \nu_{r}} = \exp \Big( \frac{- t (r - 1) (r + d - 2)}{R^{2}} \Big)
    \qquad \text{ for all } r \ge 1,
  \end{align}
  \( m_{1} = 1 \), \( m_{2} = d + 1 \) and 
  \begin{align}
    m_{r}= |I_r| = \binom{d + r - 1}{r - 1} - \binom{d + r - 3}{r - 3}
          = (2 r + d - 3) \frac{(r+d-3)!}{(d-1)!(r-1)!}
  \end{align}
  for all \( r \ge 3 \).
\end{lemma}

  Both claims follows from \cite[Chapter II.4]{MR768584}. Here a typo has to be corrected, It yields \( k + 1 \) instead of \( 2 k + 1 \) for the multiplicities of \( S^{2} \). 
  Moreover, the local Weyl law holds with $L=1$, as stated in the next lemma.

\begin{lemma}[Local Weyl law for the sphere]
  \label{lem_LocalWeylLawForTheSphere}
  The eigenfunctions of the heat kernel operator \( L_{k_{t}} \) satisfy a local Weyl law as in Assumption \ref{ass_LocalWeylLaw} with a constant \( L = 1 \).
\end{lemma}

\begin{proof}[Proof]
  Follows from \cite[Theorem 2.9]{MR2934227}. Alternatively, one can also argue elementary by invariance arguments. Since the Laplacian is invariant under any orthogonal tranformation \( T \) of \( \mathbb{R}^{d + 1} \), the same is true for its eigenspaces.
  Therefore, for any \( r \ge 1 \), both \( (Y_{r,j})_{1\leq j\leq m_r} \) and \( (Y_{r,j} \circ T)_{1\leq j\leq m_r} \) are orthonomal bases of the \( r \)-th eigenspace.
  From the uniqueness of projection kernel onto this space and the smoothness of the \( (Y_{r,j})_{1\leq j\leq m_r} \), it now follows that
  \begin{align}
    \sum_{j=1}^{m_r}  Y_{r,j}^{2}(x) = \sum_{j=1}^{m_r} Y_{r,j}^{2}(T x)
    \qquad \text{ for all } x \in \mathbb{S}^{d}.
  \end{align}
  Since any \( y \in \mathbb{S}^{d}(R) \) can be represented as \( T x \), the sum is constant.
  By integrating with respect to \( \mathbb{P}^{X} \), we obtain that it takes the value \( m_{r} \). 
\end{proof}

The heat kernel example involves three parameters, the number of observations $n$, the time parameter $t>0$ and the dimension $d$. Depending on their size and relation, we will consider the large $t$ regime, the small $t$ regime, and the large $d$ regime. The following lemma provides the relevant bounds for the relative rank \( \mathbf{r}_{r}(\Sigma) \). Interestingly, there are dimensionality effects depending on whether $d=1$, $d=2$, or $d\geq 3$, and the relative rank behaves differently for small and large $t$.

\begin{lemma}[Relative rank for the heat kernel on the sphere]
  \label{lem_RelativeRankForTheHeatKernelInDifferentRegimes}
  In the above setting, the following bounds hold.
  \begin{enumerate}[label=(\roman*)]
    \item 
      For \( d = 1 \), there is an absolute constant \( C > 0 \) such that for any \( r \ge 1 \),
      \begin{align}
        \mathbf{r}_r(\Sigma) \le C \Big( r+ \frac{1}{tr} \log(er) \Big).
      \end{align}

    \item 
      For \( d = 2 \), there is an absolute constant \( C > 0 \) such that for any \( r \ge 1 \),
      \begin{align}
        \mathbf{r}_r(\Sigma) & \le C\Big( r^{2} + \frac{1}{t}(\log(e r) + \log_+(1/t)) \Big).
      \end{align}

    \item
      For \( d \ge 3 \), there is a constant $C>0$ depending only on $d$ such that  for any \( r \ge 1 \),
      \begin{align}
        \mathbf{r}_r(\Sigma) & \le C\Big( r^{d} + \frac{r^{d-2}}{t} \log(er) + t^{-d/2} \Big).
      \end{align}
      
    \item 
      In the large \( d \) regime, under the assumption that  \( t d / R^2 \ge \log 2 \), there exists an absolute constant $ C > 0 $ such that 
      \begin{align}
        \mathbf{r}_2(\Sigma) \le C (d + 1) \exp \Big( 2 (d + 1) e^{-\frac{d t}{R^2}} \Big)
        \qquad \text{ and } \qquad 
        \frac{\mu_{2}}{g_2} \le C \Big( 1 + \frac{R^{2}}{t d} \Big).
      \end{align}
      In particular, if \( t d / R^2 \ge \log d \), then 
      \begin{align}
        \mathbf{r}_2(\Sigma) \le  C d
        \qquad \text{ and } \qquad 
        \frac{\mu_{2}}{g_2} \le C.
      \end{align}
  \end{enumerate}
\end{lemma}

The proof of Lemma \ref{lem_RelativeRankForTheHeatKernelInDifferentRegimes} is in Appendix \ref{prf_RelativeRankForTheHeatKernelInDifferentRegimes}. In addition to the relative rank, further eigenvalue expressions arise, which are bounded in the following lemma. The proof is analogous to that of Lemma \ref{lem_RelativeRankForTheHeatKernelInDifferentRegimes} and is therefore omitted.

\begin{lemma}
  \label{lem_EVExpressBound}
  In the above setting, the following bounds hold.
  \begin{enumerate}[label=(\roman*)]
    \item  For \( d \ge 1 \), there is a constant $C>0$ depending only on $d$ such that for any \( r \ge 1 \),
      \begin{align}
         \frac{\mu_{r}}{g_r} \le C \Big( 1 + \frac{1}{tr} \Big).
      \end{align}
    \item 
      For \( d \ge 5 \), there is a constant $C>0$ depending only on $d$ such that  for any \( r \ge 1 \),
      \begin{align}
        \sum_{s\neq r}\frac{m_s\mu_{r,t}\mu_{s,t}}{(\mu_{r,t}-\mu_{s,t})^2} & \le C\Big( r^{d} + \frac{r^{d-3}}{t^2} + t^{-d/2} \Big).
      \end{align}
  \end{enumerate}
\end{lemma}

Our first corollary deals with the large \( t \) exponential decay regime.

\begin{corollary}[The large $t$ regime for the sphere with fixed dimension]
  \label{cor_SDLargeTFixedD}
  Let \( d \ge 1 \), $r \ge 1$ and $t > 0$ such that $t r^2 \ge \log(er)$. Then there are constants $c_d,C_d >0$ depending only on $d$ such that the following holds. If
  \begin{align}
    \frac{r^{d} \log n}{n} \le c_d
    \qquad \text{ and } \qquad 
    t r \ge 1,
  \end{align}
  then we have
  \begin{align}
        \frac{1}{\sqrt{m_r}} 
        \Big\|\Big( \frac {\widehat{\lambda}_{j, t}-\mu_{r, t}}{\mu_{r, t}}  \Big)_{j \in I_{r}} 
          \Big\|_{2}
    \le 
        C_{d}  \Big( \sqrt{ \frac{r^{d-1} \log n}{n} } + \frac{r^{d} \log n}{n} \Big)
  \end{align}
   and 
  \begin{align}
    \inf_{O \in O(m_r)} 
    \frac{1}{\sqrt{m_r}}\Bigg\| (\widehat v_{j,t})_{j \in I_r}O - \frac{1}{\sqrt{n}}  \begin{pmatrix}
                                                           Y_{r,1}(X_1) &\cdots &  Y_{r,m_r}(X_1)\\
                                                           \vdots  & \ddots & \vdots    \\
                                                           Y_{r,1}(X_n) &\cdots & Y_{r,m_r}(X_n) 
                                                         \end{pmatrix}
    \Bigg\|_2
    \le 
     C_{d}  \sqrt{ \frac{r^{d} \log n}{n} } 
  \end{align}
  with probability at least \( 1 - 1/n^2\).
\end{corollary}

\begin{proof}[Proof of Corollary \ref{cor_SDLargeTFixedD}]
    If $t r^2 \ge \log(er)$, then $\mathbf{r}_r(\Sigma)\leq C_dr^d$ by Lemma \ref{lem_RelativeRankForTheHeatKernelInDifferentRegimes}(i)--(iii). If additionally $tr\geq 1$, then $\mu_r/g_r\leq C_d$ by Lemma \ref{lem_EVExpressBound}(i). Finally, we have $m_r\leq C_dr^{d-1}$ by Lemma \ref{lem_HeatKernelEigenvaluesAndTheirMultiplicities}. Inserting these bounds into Theorem \ref{thm_BoundsForEigenvaluesAndEigenspaces}, the claim follows.
\end{proof}

Up to the index $r$ there are $m_1+\dots+m_r\geq cr^d$ eigenvalues. Hence, Corollary \ref{cor_SDLargeTFixedD} allows us to study the leading $cn$ eigenvalues and eigenvectors of kernel matrices. Hence, our results extend the bounds from the prototypical setting of sub-Gaussian distributions and simple eigenvalues of \cite{Wahl2026PertubationExpansionsI} to the kernel-based setting.

We now turn to the small $t$ regime. In this case the leading eigenvalues of the sequence $e^{-t\nu_1},e^{-t\nu_2},\dots$ are all close to one and it makes sense to consider the difference quotient $t^{-1}(I_n-K_t)$ as approximation of $t^{-1}(I-e^{-t\Delta_{\mathbb{S}^{d}(R)}})$. As the latter converges to its generator $\Delta_{\mathbb{S}^{d}(R)}$ as $t\rightarrow 0$, this regime is important in the setting of manifold learning. The following corollary only considers the case $d\geq 3$ and $d\geq 5$. Similar statements hold for $d=1$ and $d=2$.

\begin{corollary}[The small $t$ manifold regime for the sphere with $d\geq 3$]
    \label{cor_SDsmallTFixedD}
  Let $d \ge 3$ and $n \ge 2$.
  Let $r \ge 1$ and $t > 0$ such that \( t r^{2} \log^{2/(d - 2)}(e r) \le 1 \).  Then there are constants $c_d,C_d >0$ depending only on $d$ such that the following holds. If
  \begin{align}
    \frac{\log n}{nt^{d/2+1}r} \le c_d,
  \end{align}
  then we have
    \begin{align}
          \frac{1}{\sqrt{m_{r}}} 
          \Big\| \Big( \frac{1 - \widehat{\lambda}_{j, t}}{t} - \nu_{r} \Big)_{j \in I_{r}} 
          \Big\|_{2}
    \le 
          C_d          \Big( \sqrt{\frac{r^{d-1}\log n}{nt^2}} 
              + \frac{\log n}{nt^{d/2+1}}
        + 
          t\nu_r^2\Big)
  \end{align}
  and, if additionally $d\geq 5$,
  \begin{align}\label{cor_SDsmallTFixedD:ES}
    \inf_{O \in O(m_r)} 
    \frac{1}{\sqrt{m_r}}&\Bigg\| (\widehat v_{j,t})_{j \in I_r} O- \frac{1}{\sqrt{n}}  \begin{pmatrix}
                                                           Y_{r,1}(X_1) &\cdots &  Y_{r,m_r}(X_1)\\
                                                           \vdots  & \ddots & \vdots    \\
                                                           Y_{r,1}(X_n) &\cdots & Y_{r,m_r}(X_n) 
                                                         \end{pmatrix}
    \Bigg\|_2
    \\
    &\le 
    C_d\Big(\sqrt{\frac{\log n}{nt^{d/2}}}+ \sqrt{\frac{r^{d-3}\log n}{nt^2}}+\frac{\log n}{nt^{d/2+1}r}\Big)\notag
  \end{align}
  with probability at least \( 1 - 1/n^2\).
\end{corollary}

Corollary \ref{cor_SDsmallTFixedD} generalizes and explains spectral convergence rates from manifold learning in several directions. Firstly, the leading term in the eigenvector bound provides the state-of-the-art dependence on $t$ as derived in \cite{MR4452681,TLV25}, and provides a better dependence on $t$ in the eigenvalue bound. However, the above result only gives the stochastic term of the problem. An additional bias term comes from the fact that in the kernel matrix the true heat kernel has to be replaced by e.g.~a Gaussian kernel. Secondly, our bounds are non-asymptotic and in particular explicit in the potentially large eigenvalue index $r$ and the multiplicity $m_r$. Finally, let us mention that the above leading order term in the eigenvector result also appears in a lower bound for principal component analysis, provided that one replaces the Karhunen-Loève coefficients in Lemma \ref{lem:KL} by independent standard Gaussian random variables \cite{MR4647373}.

\begin{proof}
  If \( t r^{2} \log^{2/(d - 2)}(e r) \le 1 \) and $d\geq 3$, then $\mathbf{r}_r(\Sigma)\leq C_dt^{-d/2}$ and $\mu_r/g_r\leq C_d/(rt)$ by Lemma \ref{lem_RelativeRankForTheHeatKernelInDifferentRegimes}(iii),(iv). Moreover, we have $m_r\leq C_dr^{d-1}$ by Lemma \ref{lem_HeatKernelEigenvaluesAndTheirMultiplicities}. Inserting these bounds and Lemma \ref{lem_EVExpressBound}(i),(ii) into Theorem \ref{thm_BoundsForEigenvaluesAndEigenspaces}, we get \eqref{cor_SDsmallTFixedD:ES} and 
  \begin{align}
      \frac{1}{\sqrt{m_{r}}} 
        \Big\| \Big( \frac{1 - \widehat{\lambda}_{j, t}}{t} -  \frac{1 - e^{-t\nu_r}}{t}\Big)_{j \in I_{r}} 
        \Big\|_{2}
  \le C_d\frac{1}{t}\Big( \sqrt{\frac{r^{d-1}\log n}{n}} 
            + \frac{\log n}{nt^{d/2}}\Big)
  \end{align}
  with probability at least \( 1 - 1/n^2\). 
  Hence, the claim follows from
  \begin{align}
      \Big|\frac{1 - e^{-t\nu_r}}{t} -\nu_r\Big|\leq \frac{t\nu_r^2}{2},
  \end{align}
  which is standard Taylor remainder bound for the exponential function.
\end{proof}

Finally, we turn to the large \( d \) regime. Since the multiplicities depend on $d$, interesting dimension effects arise here. As an illustration, we only consider the case \( r = 2 \), in which case $I_2=\{2,\dots,d+2\}$ with $m_2=d+1$ and the eigenfunctions $\phi_2,\dots,\phi_{d+2}$ are the normalized $d+1$ different coordinate projections. The following corollary requires $t$ to be larger than a constant times $\log d$. Otherwise the relative rank has an exponential dependence on $d$.

\begin{corollary}[The large $ d $ regime for the sphere with $r=2$]
  Let $n\geq 2$, $ d \ge 1$, and $t > 0$ be such that \( t d / R^2 \ge \log d \).
  Then there are absolute constants $c, C >0$ such that if $d \le c n / \log n$, then we have
  \begin{align}
      \frac{1}{\sqrt{d + 1}}
      \Big(\sum_{j = 2}^{d + 2} \Big(\frac{\widehat{\lambda}_{j, t}-\mu_{2, t}}{\mu_{2, t}} \Big)^2 \Big)^{1/2}
  \le 
      C  \sqrt{\frac{d \log n}{n}}
  \end{align}
  and consequently
  \begin{align}
      \Big| \frac{1}{d + 1} \sum_{j = 2}^{d+2} \frac{\widehat{\lambda}_{j, t}-\mu_{2, t}}{\mu_{2, t}}  \Big|
  \le 
      C \sqrt{\frac{d \log n}{n}},
  \end{align}
  as well as
  \begin{align}
     \inf_{O \in O(d+1)} 
  \frac{1}{\sqrt{d+1}}\Bigg\| (\widehat v_{j,t})_{j\in I_2}O - \frac{1}{\sqrt{n}}  \begin{pmatrix}
                                                         Y_{2,1}(X_1) & \cdots & Y_{2,d+1}(X_1) \\ 
                                                         \vdots   & \ddots & \vdots   \\
                                                         Y_{2,1}(X_n) & \cdots & Y_{2,d+1}(X_n)
                                                       \end{pmatrix} 
  \Bigg\|_2
  \le 
  C \sqrt{ \frac{d\log n}{n}
                            } 
  \end{align}
  with probability at least $1 - 1/n^{2}$.
\end{corollary}

\begin{proof}
    If \( t d / R^2 \ge \log d \), then $\mathbf{r}_2(\Sigma)\leq Cd$ and $\mu_2/g_2\leq C$ by Lemma \ref{lem_RelativeRankForTheHeatKernelInDifferentRegimes}(v). Moreover, we have $m_r=d+1$ by Lemma \ref{lem_HeatKernelEigenvaluesAndTheirMultiplicities}. Inserting these bounds and $I_2=\{2,\dots,d+2\}$ into Theorem \ref{thm_BoundsForEigenvaluesAndEigenspaces}, the claim follows.
\end{proof}


\subsection{Example 2: Truncated wavelet priors}
\label{ssec_TruncatedWaveletPriors}

In Bayesian non-parametrics Gaussian process priors play an important role, see, e.g., \citet{GhosalvdVaart2017FundamentalsOfBayes}.
Here we consider a truncated wavelet prior on $[0,1]$
\begin{align}
  \label{eq_Expls_WaveletPrior}
  W & = \sum_{r = 0}^{R} 
           \sum_{j = 0}^{2^{r} - 1} \sqrt{\mu_{r}} Z_{rj} \psi_{rj}.
\end{align}
where \( Z_{r j}  \), \( r, j \ge 0 \) are independent standard Gaussian random variables and \( (\psi_{r, j})_{r \ge 0, 0 \le j \le} \\ _{2^{r} - 1} \) is an orthonormal basis of \( L^{2}([0, 1]) \) of continuous, compactly supported wavelets with \( r \) referring to the resolution level and \( j \) to the dilation.
Such a basis can be chosen such that for numerical constant \( 0 < c \le C < \infty \), 
\begin{align}
  c 2^{r / 2} \le \| \psi_{r j} \|_{\infty} \le C 2^{r/2}
  \qquad \text{ for all } r \ge 0, 0 \le j \le 2^{r} - 1, 
\end{align}
stemming from the  dilation of a mother wavelet, and such that groupwise, the functions satisfy a bounded overlap property, i.e., there exists a constant \( N \in \mathbb{N} \) such for all \( r \ge 0 \) and \( x \in [0, 1] \)
\begin{align}
  \label{eq_Expls_BoundedOverlap}
  | \{ j: \psi_{r j}(x) \ne 0  \}| \le N,
\end{align}
see Chapter E.3 in \cite{GhosalvdVaart2017FundamentalsOfBayes} or \cite{CohenEtal1993WaveletsOnTheInterval} and \cite{GineNickl2016Foundations}.
The covariance kernel of the Gaussian process \( W \) is then given by
\begin{align}
  k(x, y) = \sum_{r = 0}^{R} \sum_{j = 0}^{2^{r} - 1} \mu_{r} \psi_{r j}(x) \psi_{r j}(y), 
  \qquad x, y \in [0, 1].
\end{align}
We note that our arguments rely essentially only on the existence of a Mercer representation of the kernel \( k \).
Consequently, one could equally well work with the standard Haar wavelets on the unit interval, which admit a more explicit representation but are not continuous. 
To preserve the clean framework of Section \ref{sec_SpectralConcentrationForKernelMatrices}, however, we adopt the construction of \cite{CohenEtal1993WaveletsOnTheInterval}.

The wavelet prior example is motivated by the application in \cite{StankewitzSzabo24}, which considers a Gaussian process regression based on data
\begin{align}
  Y_{i} = f(X_{i}) + \varepsilon_{i}, \qquad i = 1, \dots, n
\end{align}
with, e.g., \( \varepsilon_{i} \sim N(0, \sigma^{2}) \), \( X_{i} \sim \text{Unif}[0, 1] \) i.i.d. and \( f \) endowed with a centered Gaussian process prior with a kernel \( k \) that satisfies a Mercer expansion as in Lemma \ref{lem_Mercer} for $ \mathcal{X}=[0, 1] $.
\cite{StankewitzSzabo24} then derives minimax optimal contraction rates for an approximate posterior process in which the computationally expensive inversion of the \( n \times n \)-matrix \( n K + \sigma^{2} I_n \) in the true posterior GP with mean and covariance function
\begin{align}
  x       & \mapsto k(x, X)^{\top} (n K + \sigma^{2} I_n)^{-1} Y,
  \qquad \qquad \qquad \qquad \qquad \ \ x \in [0, 1] \\
  (x, x') & \mapsto k(x, x') - k(x, X)^{\top} (n K + \sigma^{2} I_n)^{-1} k(x', X),
  \qquad x, x' \in [0, 1]
  \notag 
\end{align}
is replaced by \( m = m(n) \ll n \) steps of a conjugate-gradient iteration.
The derivation of these results crucially depends on the numerical stability of \( K \), which is expressed in terms of spectral concentration for its singular value decomposition.
Importantly, this actually requires relative bound of the form
\begin{align}
  \label{eq_Expls_UniformControlOverSVD}
  \max_{1\leq j\leq m}\frac{|\widehat{\lambda}_{j} - \lambda_{j}|}{\lambda_{j}} = o(1)
\end{align}
with high probability. However, so far such results have largely been available under the more restrictive assumptions that the eigenvalues \( \lambda_{1} > \lambda_{2} > \dots \) are simple and that the eigenfunctions satisfy a uniform boundedness condition of the form $\sup_{j \ge 1} \mathbb{E} |\phi_{j}(X)|^{p} < \infty$ for \( p > 4 \) sufficiently large.
While some extensions to settings with multiple eigenvalues have been considered, these have not yet been developed in the same level of generality or detail.

In the wavelet prior from Equation \eqref{eq_Expls_WaveletPrior}, for any \( r = 0, 1, \dots R \), the multiplicity of \( \mu_{r} \) is \( m_{r} = 2^{r}  \).
Anticipating that the truncation level \( R \) will have to be chosen depending on the sample size \( n \), we won't have access to a uniform \( p \)-th moment bound either as \( \| \psi_{r j} \|_{\infty} \ge c 2^{r/2} \).
Still, our results from Section \ref{ssec_MainResults} are applicable in this more complicated setting.
Indeed, for a sample \( X_{1}, \dots, X_{n} \) of i.i.d.~\(\text{Unif} [0, 1]\)-distributed random variables, Assumptions \ref{ass_data}--\ref{ass_Mercer} are satisfied for the wavelet kernel and because of the bounded overlap property in Equation \eqref{eq_Expls_BoundedOverlap} the wavelet kernel also satisfies Assumption \ref{ass_LocalWeylLaw}, as can be seen as follows 
\begin{align}
        \sum_{j = 0}^{2^{r} - 1} \psi_{r j}(x)^{2} 
  & \le N \sup_{j \in I_{r}} \| \psi_{r j} \|_{\infty}^{2} 
    \le N C 2^{r}.
\end{align}

In order to model the smoothness of the prior, following the exposition in \cite{GhosalvdVaart2017FundamentalsOfBayes}, we consider eigenvalues \( \mu_{r} = 2^{-r (2 \alpha - 1)} \), \( r \ge 0 \), for some \( \alpha > 1/2 \) and \( R = R_{\alpha} \) such that \( 2^{R_{\alpha}} \) is of the size of the classical non-parametric truncation level \( n^{1 / (2 \alpha + 1)} \), i.e. 
\begin{align}
  R_{\alpha} = \big\lceil \log_{2} n^{1 / (2 \alpha + 1)} \big\rceil.
\end{align}
In this setting, we obtain the following bounds for the relative rank quantities.

\begin{lemma}[Relative rank for wavelet priors]
  \label{lem_RelativeRankForWaveletPriors}
  For any fixed \( r = 0, 1, \dots, R_{\alpha} \), we have
  \begin{align}
        \mathbf{r}_{r}(\Sigma) 
                    \le C \begin{cases}
                            2^{(2-2\alpha) R_{\alpha}} 2^{(2\alpha - 1) r}, & \alpha \in (1/2, 1), \\
                            (R_{\alpha} - r + 1) 2^{r},                     & \alpha =         1, \\
                            2^{r},                                          & \alpha >         1,
                          \end{cases}
  \end{align}
  and \(\mu_{r}/g_{r} \le C\) with a constant \( C > 0 \) depending only on \( \alpha \).  
\end{lemma}

\noindent Its derivation is in Section \ref{prf_RelativeRankForWaveletPriors}.
An application of the eigenvalue statement in Theorem \ref{thm_BoundsForEigenvaluesAndEigenspaces} then leads to the following bounds.

\begin{corollary}[Individual eigenvalues of truncated wavelet priors]
  \label{cor_IndividualEigenvaluesOfTruncatedWaveletPriors}
  Let \( \alpha > 1/2 \).
  Then for any \( r = 0, 1, \dots, R_{\alpha} \), the empirical eigenvalues and eigenvectors corresponding to \( \mu_{r} \) satisfy
  \begin{align}
       \frac{1}{2^{r/2}}  
       \Big\| \Big( \frac{\widehat{\lambda}_{j} - \mu_{r}}{\mu_{r}} \Big)_{j \in I_{r}} \Big\|_{2}  
    \le 
        C \sqrt{\frac{\log n}{n^{2 \alpha / (2 \alpha + 1)}}}
    \end{align}
    and
    \begin{align}
    \inf_{O \in O(2^r)} 
    \frac{1}{2^{r/2}} \bigg\| (\widehat{v}_{j})_{j \in I_{r}} O -\frac{1}{\sqrt{n}} \begin{pmatrix}
                                                                    \psi_{r j}(X_{1}) \\
                                                                    \vdots            \\
                                                                    \psi_{r j}(X_{n}) \\
                                                                  \end{pmatrix}_{j \in I_{r}}
                      \bigg\|_{2}
    \le 
    C \sqrt{\frac{\log n}{n^{2 \alpha / (2 \alpha + 1)}}}.
  \end{align}
  with probability at least \( 1 - n^{-2} \).
\end{corollary}

\begin{proof}[Proof]
  For any \( \alpha > 1/2 \) and \( r = 0, \dots, R_{\alpha} \), Lemma \ref{lem_RelativeRankForWaveletPriors} implies that
  \begin{align}
    \frac{\mu_{r}}{g_{r}} \mathbf{r}_{r}(\Sigma) \frac{\log n}{n}
    & \le 
    C 2^{J_{\alpha}} \frac{\log n}{n} 
    \le 
    C n^{1 / (2 \alpha + 1)} \frac{\log n}{n} 
    \le 
    C n^{-(2 \alpha) / (2 \alpha + 1)} \log n,
  \end{align}
  which is bounded by an arbitrarily small constant \( c > 0 \) for \( n \) sufficiently large.
  The statements now follow from Theorem \ref{thm_BoundsForEigenvaluesAndEigenspaces} with \( \gamma = 2 \).
\end{proof}

Motivated by the application in \cite{StankewitzSzabo24} and the requirement in Equation \eqref{eq_Expls_UniformControlOverSVD}, we also formulate a result for the eigenvalues showing that for any \( \alpha > 1/2 \), the techniques from Section \ref{sec_SpectralConcentrationForKernelMatrices} guarantee uniform control along the whole singular value decomposition of the truncated wavelet prior.

\begin{corollary}[Uniform control for truncated wavelet priors]
  \label{cor_UniformControlForTruncatedWaveletPriors}
  In the situation above, for any \( \alpha > 1/2 \),
  \begin{align}
        \max_{0 \le r \le R_{\alpha}} 
        \max_{j \in I_{r}} \frac{| \widehat{\lambda}_{j} - \mu_{r} |}{\mu_{r}} 
    \le 
        C \sqrt{n^{-\frac{2 \alpha - 1}{2 \alpha + 1}} \log n} 
  \end{align}
  with probability at least \( 1 - n^{-1} \).
\end{corollary}

\begin{proof}[Proof]
  This follows simply by estimating the maximum difference over the eigenblock against the \( \ell^{2} \)-norm of the differences, using the worst case bound from Corollary \ref{cor_IndividualEigenvaluesOfTruncatedWaveletPriors} and accounting for the added multiplicity \( m_{r} = 2^{r} \) on both sides. 
  Finally the uniformity in \( r = 0, \dots, R_{\alpha} \) follows with a union bound.
\end{proof}



\section{Proofs of the probabilistic estimates}
\label{ssec_ProofsOfTheProbabilisticEstimates}

In this section, we provide a detailed proof of Proposition \ref{prp_BoundsOnTheLinearAndQuadraticTerm} using moment inequalities for U-statistics. 
Here, Assumption \ref{ass_LocalWeylLaw} on the growth behavior of the eigenfunctions are instrumental in deriving probabilistic bounds on the pertubation terms. The following arguments are also able to cover other assumptions that are frequently satisfied by the \( (\phi_{j})_{j \ge 1} \) as discussed in Section \ref{ssec_ProblemSetup}.

\begin{proof}[\normalfont \textbf{Proof of Proposition \ref{prp_BoundsOnTheLinearAndQuadraticTerm}} (Probabilistic bounds)]
  \label{prf_BoundsOnTheLinearAndQuadraticTerm}
  Inserting Lemma~\ref{lem:KL}, the term \( \| P_{r} E P_{r} \|_{2}^2 \) can we written as
    \begin{align}
        \| P_{r} E P_{r} \|_{2}^{2} 
    &= 
        \mu_{r}^{2} \sum_{ j, k \in I_{r} } 
                    \Big( \frac{1}{n} \sum_{ i = 1 }^{n} 
                          \phi_j(X_i) \phi_k(X_i) - \delta_{j k}
                    \Big)^{2}\\
    &= \frac{1}{n^2}\sum_{i=1}^n \mu_r^2\sum_{ j, k \in I_{r} } (\phi_j(X_i) \phi_k(X_i) - \delta_{j k})^2\\
    &+\frac{2}{n^2}\sum_{1\leq i<i^\prime\leq n} \mu_r^2\sum_{ j, k \in I_{r} }(\phi_j(X_i) \phi_k(X_i) - \delta_{j k})(\phi_j(X_{i^\prime}) \phi_k(X_{i^\prime}) - \delta_{j k}).
  \end{align}
  We introduce the function
  \begin{align}
      h(x,y)=\mu_r^2\sum_{ j, k \in I_{r} }(\phi_j(x) \phi_k(x) - \delta_{j k})(\phi_j(y) \phi_k(y) - \delta_{j k}),\qquad x,y\in\mathcal{X}.
  \end{align}
  Then, $h$ is a symmetric and degenerate kernel (meaning that $h(x,y)=h(y,x)$ for all $x,y\in\mathcal{X}$ and $\mathbb{E}h(X,y)=0$ for all $y\in\mathcal{X}$) with
  \begin{align}\label{eq:bound:kernel}
      \sup_{x,y\in\mathcal{X}}|h(x,y)|\leq \mu_r^2\sup_{x\in\mathcal{X}}\sum_{ j, k \in I_{r} }(\phi_j(x) \phi_k(x) - \delta_{j k})^2\leq 2L^2m_r^2\mu_r^2,
  \end{align}
  as can be seen from the Cauchy-Schwarz inequality and Assumption \ref{ass_LocalWeylLaw}. We now have 
  \begin{align}
      \| P_{r} E P_{r} \|_{2}^{2}&=\frac{1}{n^2}\sum_{i=1}^n h(X_i,X_i)+\frac{2}{n^2}\sum_{1\leq i<i^\prime\leq n} h(X_i,X_{i^\prime}).
  \end{align}
  By \eqref{eq:bound:kernel}, we have
  \begin{align}\label{eq:HD:1}
      0\leq \frac{1}{n^2}\sum_{i=1}^n h(X_i,X_i)\leq 2L^2\frac{m_r^2\mu_r^2}{n}\qquad a.s.
  \end{align}
  Moreover, by symmetrization (see \cite{SCK19} or \cite[Theorem 3.1]{PG99} for a result with slightly worse constants) and the Bonami inequality (\cite[Theorem 3.22]{PG99}), we have 
    \begin{align}
        \forall p\geq 1,\qquad&\mathbb{E}^{1/p}\Big|\frac{1}{\binom{n}{2}^{1/2}}\sum_{1\leq i<i^\prime\leq n}h(X_{i},X_{i^\prime})\Big|^p\\
        &\leq 4\mathbb{E}^{1/p}\Big|\frac{1}{\binom{n}{2}^{1/2}}\sum_{1\leq i<i^\prime\leq n}\epsilon_{i}\epsilon_{i^\prime}h(X_{i},X_{i^\prime})\Big|^p\\
        &\leq 4(p-1)\mathbb{E}^{1/p}\Big|\frac{1}{\binom{n}{2}}\sum_{1\leq i<i^\prime\leq n}h(X_{i},X_{i^\prime})^2\Big|^{p/2}.
    \end{align}
    Inserting \eqref{eq:bound:kernel}, we get
    \begin{align}
       \forall p\geq 1,\qquad\mathbb{E}^{1/p}\Big|\frac{1}{\binom{n}{2}^{1/2}}\sum_{1\leq i<i^\prime\leq n}h(X_{i},X_{i^\prime})\Big|^p\leq 8L^2m_r^2\mu_r^2p,
    \end{align}
    meaning that
    \begin{align}\label{eq:HD:2}
       \Big\|\frac{2}{n^2}\sum_{1\leq i<i^\prime\leq n}h(X_{i},X_{i^\prime})\Big\|_{\psi_1}\leq 8\sqrt{2}L^2\frac{m_r^2\mu_r^2}{n}.
    \end{align}
    Combining \eqref{eq:HD:1} and \eqref{eq:HD:2}, we arrive at
    \begin{align}
        \big\|\| P_{r} E P_{r} \|_{2}^2\big\|_{\psi_1}&\leq  16\sqrt{2}L^2\frac{m_r^2\mu_r^2}{n}.
    \end{align}
    This gives claim (i). Inserting again Lemma~\ref{lem:KL}, the second term \( \| P_{r} E | R_{r} |^{1/2} \|_{2}^2 \) can we written as
    \begin{align}
        \| P_{r} E | R_{r} |^{1/2} \|_{2}^2
    &= 
        \mu_{r}\sum_{s\neq r}\frac{\mu_s}{|\mu_r-\mu_s|} \sum_{ j\in I_{r}, k \in I_{s}  } 
                    \Big( \frac{1}{n} \sum_{ i = 1 }^{n} 
                          \phi_j(X_i) \phi_k(X_i)
                    \Big)^{2}\\
    &= \frac{1}{n^2}\sum_{i=1}^n \mu_{r}\sum_{s\neq r}\frac{\mu_s}{|\mu_r-\mu_s|} \sum_{ j\in I_{r}, k \in I_{s}  } \phi_j(X_i)^2 \phi_k(X_i)^2
    \notag
    \\
    &+\frac{2}{n^2}\sum_{1\leq i<i^\prime\leq n} \mu_{r}\sum_{s\neq r}\frac{\mu_s}{|\mu_r-\mu_s|} \sum_{ j\in I_{r}, k \in I_{s} }\phi_j(X_i) \phi_k(X_i)\phi_j(X_{i^\prime}) \phi_k(X_{i^\prime}).
    \notag 
  \end{align}
  This time, we introduce the function
  \begin{align}
      h(x,y)=\mu_{r}\sum_{s\neq r}\frac{\mu_s}{|\mu_r-\mu_s|} \sum_{ j\in I_{r}, k \in I_{s}  } \phi_j(x)\phi_k(x) \phi_j(y) \phi_k(y),\qquad x,y\in\mathcal{X},
  \end{align}
  which is again a symmetric and degenerate kernel satisfying
  \begin{align}
      \sup_{x,y\in\mathcal{X}}|h(x,y)|\leq  L^2m_r\mu_r\sum_{s\neq r}\frac{m_s\mu_s}{|\mu_r-\mu_s|},
  \end{align}
  as can be seen from the Cauchy-Schwarz inequality and Assumption \ref{ass_LocalWeylLaw}. Using the decomposition
  \begin{align}
      \| P_{r} E | R_{r} |^{1/2} \|_{2}^2&=\frac{1}{n^2}\sum_{i=1}^n h(X_i,X_i)+\frac{2}{n^2}\sum_{1\leq i<i^\prime\leq n} h(X_i,X_{i^\prime}),
  \end{align}
  the claim (ii) now follows from proceeding as in the case of the first term. Claim (iii) follows analogous to claim (ii).
\end{proof}

Beyond the bounds on the perturbation terms, we also need a probabilistic estimate for \( \delta_{r} \) to employ our perturbation bounds in the setting of Section \ref{sec_SpectralConcentrationForKernelMatrices}.
Here, we formulate the following non-asymptotic bound.

\begin{proposition}[Non-asymptotic bound on \( \delta_{r} \)]
  \label{prp_NonAsymptoticBoundOnDelta_r}
  Assume that Assumptions \ref{ass_data}--\ref{ass_LocalWeylLaw} hold. Then, for any \(\gamma > 0 \), we have
  \begin{align}
          \delta_{r} 
    & \le 
          \sqrt{4 \gamma L \frac{\mu_{r}}{g_{r}} \mathbf{r}_r(\Sigma) \frac{\log n}{n} } 
        + 
          \frac{ 4\gamma }{ 3 } L \mathbf{r}_{r}( \Sigma ) \frac{ \log n }{n}
  \end{align}
  with probability at least \( 1 - 4 \mathbf{r}_{r}( \Sigma ) n^{- \gamma} \).
  In particular, \( \delta_{r}' < 1 / 4 \) with probability at least \( 1 - 4 \mathbf{r}_{r}( \Sigma ) n^{- \gamma} \) if \( L \gamma \mu_{r} / g_{r}  \mathbf{r}_r(\Sigma) \le c n / \log n \) for some sufficiently small constant $c > 0$.
\end{proposition}

\begin{proof}
  [\normalfont \textbf{Proof of Proposition \ref{prp_NonAsymptoticBoundOnDelta_r}}
   (Non-asymptotic bound on \( \delta_{r} \))]
  \label{prf_BoundOnDelta_r}
  The claim follows from (operator) Bernstein inequality. We set
  \begin{align}
         X': 
    & =   
         ( | R_{r} |^{ 1 / 2 } + g_{r}^{-1/2} P_{r} ) k(X, \cdot) 
    \\
    & =   
          \sum_{s\neq  r}\sum_{ j \in I_{s} }
          \sqrt{ \frac{ \mu_{s} }{ | \mu_{s} - \mu_{r} | } } 
          \phi_j(X) u_{j} 
        + 
          \sum_{ j \in I_{r} }\sqrt{ \frac{ \mu_{r} }{ g_{r} } } 
          \phi_j(X) u_{j}.
    \notag
  \end{align}
  By Assumption \ref{ass_LocalWeylLaw}, we
  immediately obtain
  \begin{align}
          \| X' \|_{\mathcal{H}_k}^{2}&=\sum_{s \neq r}
          \frac{ \mu_{s} }{ | \mu_{s} - \mu_{r} | }
          \sum_{ j \in I_{s} }
          \phi_j^2(X)+\frac{ \mu_{r} }{ g_{r} }  
          \sum_{ j \in I_{r} }
          \phi_j^2(X)
    \\
    & \le 
          L\sum_{s \neq r}
          \frac{ m_s\mu_{s} }{ | \mu_{s} - \mu_{r} | }+L\frac{m_r \mu_{r}}{g_{r}}
      =   
          L \mathbf{r}_{r}(\Sigma).
    \notag
  \end{align}
  For \( \xi: = X' \otimes X' - \mathbb{E} ( X' \otimes X' ) \), this yields
  \begin{align}
    \label{eq_XiOperatorBounds}
            \| \xi \|_\infty 
    & \le
            \| X' \|_{\mathcal{H}_k}^{2} + \mathbb{E}\| X' \|_{\mathcal{H}_k}^{2}
      \le
            2 L \mathbf{ r }_{ r }( \Sigma ),
    \\
            \| \mathbb{E} \xi^{2} \|_{ \infty } 
    & \le 
            \| \mathbb{E} ( X' \otimes X' )^{2} \|_{ \infty }
    = 
            \| \mathbb{E} \| X' \|_{\mathcal{H}_k}^{2} ( X' \otimes X' ) \|_{ \infty }
    \le
            L \Big( \frac{\mu_{r - 1}}{\mu_{r - 1} - \mu_{r}} \lor \frac{\mu_{r}}{g_{r}} \Big) \mathbf{r}_{r}(\Sigma)
    \notag
    \\
    & \le   2 L \frac{\mu_{r}}{g_{r}} \mathbf{r}_{r}(\Sigma), 
    \notag
  \end{align}
  since \( \mu_{r - 1} / (\mu_{r - 1} - \mu_{r}) \le 1 + \mu_{r} / (\mu_{r - 1} - \mu_{r}) \le 2 \mu_{r} / g_{r}\).
  Finally,
  \begin{align}
    \label{eq_XiTraceBound}
          \tr \mathbb{E} \xi^{2} 
    & \le
          \tr \mathbb{E} ( X' \otimes X' )^{2} 
    = 
          \mathbb{E}\| X' \|_{\mathcal{H}_k}^{4}\notag
    \le 
          L \mathbf{r}_{r}(\Sigma) \mathbb{E}\| X' \|_{\mathcal{H}_k}^{2}\\
    & =
          L \mathbf{r}_{r}(\Sigma)
          \Big( 
                \sum_{s < r}
                \frac{ m_{ s } \mu_{s} }{ \mu_{s} - \mu_{r} }
              + 
                \frac{ m_{r} \mu_{r} }{ g_{r} }
              +
                \sum_{s > r} 
                \frac{ m_{ s } \mu_{s} }{ \mu_{r} - \mu_{s} }
          \Big)
    \le 
          L \mathbf{r}_{r}(\Sigma)^{2}.
  \notag
  \end{align}
  Writing \( \delta_{r} = n^{-1} \sum_{ i = 1 }^{ n } \xi_{ i } \) with \( \xi_{1}, \dots, \xi_n \)
  i.i.d. copies of \( \xi \) and applying Lemma 5 from \cite{MR3629418} with, in the notation there,
  \begin{align}
    R: = 2 L \mathbf{ r }_{ r }( \Sigma ), 
    \qquad 
    V: = 2 n L \frac{\mu_{r}}{g_{r}} \mathbf{r}_{r}(\Sigma)
    \qquad \text{and} \qquad 
    D: = \mathbf{r}_{r}( \Sigma ) 
  \end{align}
  finishes the proof.
\end{proof}


\section{Higher-order spectral perturbation bounds}
\label{sec_HigherOrderSpectralPerturbationBounds}

In this section, we present the spectral perturbation bounds that are at the basis of our main results in the kernel setting. We follow the strategy in \cite{Wahl2026PertubationExpansionsI}, which in itself extends a line of research \cite{MR3300524,MR3573302,MR4102689,JirakWahl2018PerturbationBoundsUnderRelativeGap,JirakWahl2023RelativePerturbationBounds}. 
Although this part of our proof follows the steps of \cite{Wahl2026PertubationExpansionsI}, we need several modifications to account for multiplicities and the associated fact that individual empirical eigenvalues no longer admit a perturbation series in this case. 

\subsection{Notation}
\label{ssec_Notation}

The results presented here are more general than the kernel setting.
However, we keep a notation that is suggestive of Section \ref{sec_SpectralConcentrationForKernelMatrices}.
Consider therefore two general positive, self-adjoint, compact operators \( \Sigma \) and \( \widehat{\Sigma} \) on a separable Hilbert space \( \mathcal{H} \).
We consider $\widehat \Sigma$ as an approximation of $\Sigma$ and define the perturbation operator $E = \widehat \Sigma - \Sigma$ so that $\widehat \Sigma = \Sigma + E$.
By the spectral theorem, there is a sequence $ \lambda_1 \ge \lambda_2 \ge \dots $ of eigenvalues of $\Sigma$ (converging to zero), together with an orthonormal system of eigenvectors $u_1, u_2, \dots$ such that $\Sigma = \sum_{j\geq 1}\lambda_j u_j \otimes u_j$. 
Similarly, there exists a sequence $ \widehat{\lambda}_1 \ge \widehat{\lambda}_2 \ge \dots $ of eigenvalues of $\widehat \Sigma$, together with an orthonormal system of eigenvectors $ \widehat{u}_1, \widehat{u}_2, \dots $ such that $\widehat{\Sigma} = \sum_{j \ge 1} \widehat{\lambda}_j \widehat u_j \otimes \widehat u_j$.

For notational convenience, we assume that \( (u_{j})_{j \ge 1} \) and \( (\widehat{u}_{j})_{j \ge 1} \) are orthonormal bases of \( \mathcal{H} \). This allows us to cover the kernel setting from Section \ref{sec_SpectralConcentrationForKernelMatrices}, as can be seen as follows. First, this is already true for \( (u_{j})_{j \ge 1} \), as the reproducing Hilbert space $\mathcal{H}$ in \eqref{eq:RKHS} only includes the eigenvectors corresponding to the positive eigenvalues of $L_k$ and $\Sigma$. Moreover, in Section \ref{sec_SpectralConcentrationForKernelMatrices}, the empirical covariance operator $\widehat{\Sigma}$ is a finite rank operator, so that the finite orthonormal system of eigenvectors corresponding to the non-zero eigenvalues can be extended to an orthonormal basis of $\mathcal{H}$.

As before, we consider the non-increasing sequence \( (\mu_{r})_{r \ge 1} \) of distinct eigenvalues of \( \Sigma \), \( I_{r} = \{ j \ge 1: \lambda_{j} = \mu_{r} \} \) and \( m_{r}=|I_r| \) the multiplicity of \( \mu_{r} \). For any \( r \ge 1 \), we abbreviate \( P_{r} = \sum_{j \in I_{r}} u_{j} \otimes u_{j} \) and \( \widehat{P}_{r} = \sum_{j \in I_{r}} \widehat{u}_{j} \otimes \widehat{u}_{j} \).

For $r\geq 1$, we introduce the eigengap  
\begin{align} g_{r} = (\mu_{r - 1} - \mu_{r}) \land (\mu_{r} - \mu_{r + 1})\end{align} 
and
\begin{align}
  \delta_{r} &: = \| ( | R_{r} |^{ 1 / 2 } + g_{r}^{ - 1 / 2 } P_{r} ) E 
                     ( | R_{r} |^{ 1 / 2 } + g_{r}^{ - 1 / 2 } P_{r} )
                  \|_{ \infty }, 
  \\
  \delta_{r}' &: =              \| |R_{r}|^{1/2} E |R_{r}|^{1/2} \|_{\infty}
              \lor g_{r}^{-1/2} \| |R_{r}|^{1/2} E  P_{r}        \|_{\infty}
              \lor g_{r}^{-1}   \|  P_{r}        E  P_{r}        \|_{\infty},
  \notag
\end{align}
where the perturbation \( E \) is weighted based on the reduced resolvent
\begin{align}
  R_{r}: = \sum_{ s \ne r } ( \mu_{s} - \mu_{r} )^{-1} P_{s}.
\end{align}
By elementary norm inequalites we have $\delta_r\leq 2\delta_r^\prime$ and  $\delta_r^\prime\leq \delta_r$, as well as $\delta_r\leq \|E\|_\infty/g_r$.


\subsection{Perturbation bounds for eigenvalues and spectral projectors}
\label{ssec_PerturbationSeriesForSpectralProjectors}

In \cite{Wahl2026PertubationExpansionsI} it has been shown that if $\lambda_j=\mu_r$ is a simple eigenvalue with $m_r = 1$, then the \( j \)-th perturbed eigenvalue \( \widehat{\lambda}_{j} \) and its corresponding eigenprojector \( \widehat{u}_{j} \otimes \widehat{u}_{j} \) admit Taylor expansions in the perturbation \( E \) provided that $\delta_r < 1/2$ (or $\delta_r' < 1/4$). This itself extends the classical perturbation bound stating that the \( j \)-th perturbed eigenvalue \( \widehat{\lambda}_{j} \) and its corresponding eigenprojector \( \widehat{u}_{j} \otimes \widehat{u}_{j} \) admit Taylor expansions in the perturbation \( E \) provided that \( \| E \|_{\infty} < g_r / 2 \), see, e.g.,~Theorem 3.9 in \cite{MR1335452}. The goal of this section is to show that extensions of \cite{Wahl2026PertubationExpansionsI} exist in the case of  multiple eigenvalues. This requires a new statement for the expansions of eigenvalues, since the \( \widehat{\lambda}_{j} \), \( j \in I_{r} \) now do not individually possess expansions anymore. 

For $r\geq 1$, we introduce the Taylor coefficients 
\begin{align}
  P_{r}^{ (n) }: = ( - 1 )^{ n + 1 } 
                   \sum_{ \substack{ k_{1},  \dots,  k_{ n + 1 } \ge 0 \\
                                     k_{1} + \dots + k_{ n + 1 } =   n
                                   }
                        }
                        R_{r}^{ ( k_{1} ) } E \dots E R_{r}^{ ( k_{ n + 1 } ) }, 
\end{align}
where \( R_{r}^{ ( 0 ) } = - P_{r} \) and
\( R_{r}^{ ( k ) } = R_{r}^{ k } \) for \( k \ge 1 \).
For instance, we have
\begin{align}
  P_{r}^{(0)} & = P_{r}, \qquad 
  P_{r}^{(1)}   = - P_{r} E R_{r} - R_{r} E P_{r} \qquad \text{ and } \qquad  \\
  \notag
  P_{r}^{(2)} & = P_r E R_r E R_r + R_r E P_r E R_r + R_r E R_r E P_r - R_r^2 E  P_r E P_r \\
              & - P_r E R_r^2 E P_r - P_r E P_r E R_r^2.
  \notag
\end{align}
The following first main result presents a projector perturbation series in the case of multiple eigenvalues. 
The proof of Theorem \ref{thm_ProjectorPerturbationSeries} is presented in Section \ref{sec_ProofsOfThePerturbationBounds}.

\begin{theorem}[Projector perturbation series]
  \label{thm_ProjectorPerturbationSeries}
  Let \( \delta_{r}' < 1 / 4 \). 
  Then, the series of Hilbert-Schmidt operators 
  $ \sum_{n = 0}^{\infty} P_{r}^{ (n) } $ converges absolutely and for any
  integer \( p \ge 1 \), 
  \begin{align}
        \| \widehat{P}_{r} - \sum_{n = 0}^{p - 1} P_{r}^{ (n) } \|_{2} 
    \le 
        4 g_{r}^{-1/2} 
        \| | R_{r} |^{1/2} E P_{r} \|_{2} 
        \frac{ ( 4 \delta_{r}' )^{p - 1} }{ (1 - 4 \delta_{r}')^{2} }.
  \end{align}
  In particular, it follows that \( \widehat{P}_{r} = \sum_{n = 0}^{\infty} P_{r}^{ (n) } \).
\end{theorem}

In the case of single eigenvalues a similar bound holds for the perturbed eigenvalues.
More precisely, for $ m_r = 1 $ and $j\geq 1$ such that $\lambda_j=\mu_r$, one can formulate a similar result for $ \widehat \lambda_j - \mu_r $ with Taylor coefficients given by $ \mu_r^{(n)} = \operatorname{tr}(P_{r}^{(n - 1)} E) + \operatorname{tr}(P_{r}^{(n)} R_{r}^{-1})$, $ n \geq 1$, where \( R_{r}^{-1} \) denotes the pseudoinverse of \( R_{r} \), see \cite{Wahl2026PertubationExpansionsI}.
In the case of multiple eigenvalues, we formulate an expansion of $ \widehat{P}_{r} (\widehat{\Sigma} - \mu_r) $.
For $ n \geq 1 $, we define
\begin{align}
  Q_{r}^{(n)}: = \sum_{l = 0}^{n-1} P_{r}^{(l)} E                  P_{r}^{(n - 1 - l)} 
               + \sum_{l = 0}^{n}   P_{r}^{(l)} (\Sigma - \mu_{r}) P_{r}^{(n     - l)}.
\end{align}
For instance, we have
\begin{align}
  Q_{r}^{(1)} &= P_rEP_r\\
  Q_{r}^{(2)} &= - P_rEP_rER_r - P_rER_rEP_r - R_rEP_rEP_r.
  \notag
\end{align}
If $ m_r = 1 $, then we have $ \operatorname{tr}(Q_r^{(n)}) = \mu_r^{(n)} $, \( n \ge 1 \), as can
be seen from the identity $ P_r^{(n)} = \sum_{l = 0}^n P_r^{(l)} P_r^{(n - l)} $, $ n \ge 0 $.
In the generalized setting, we have

\begin{theorem}[Eigenvalue perturbation series]
  \label{thm_EigenvaluePerturbationSeries}
    Let \( \delta_{r}' < 1 / 4 \). Then, for any \( p \ge 1 \), 
  \begin{align}
          \| \widehat{P}_{r} (\widehat{ \Sigma } - \mu_r) - \sum_{n = 1}^{p} Q_{r}^{(n)} \|_{2} 
    & \le
          8 \big( \| P_rEP_r \|_{2} 
                + g_r^{1/2} \| |R_{r}|^{1/2} E P_r \|_{2}\delta_r^\prime   
            \big)
            \frac{(p + 1) (4 \delta_r')^{p - 1}}{(1 - 4 \delta_r')^2}.
    \notag
  \end{align}
\end{theorem}

Given the Taylor expansions in Theorems \ref{thm_ProjectorPerturbationSeries} and \ref{thm_EigenvaluePerturbationSeries}, one might ask what are the best perturbation bounds that can be derived from these Taylor appoximations.

\begin{corollary}[Eigenvalue perturbation bound]
  \label{cor_EigenvaluePerturbationBound}
  If \( \delta_r' < 1 / 4 - \varepsilon \) for some $\epsilon\in(0,1/4)$, then
  \begin{align}
          \|(\widehat\lambda_j-\mu_r)_{j\in I_r}\|_{2}
    \leq
          C_{\varepsilon} ( \| P_{r} E P_{r} \|_{2} + g_r^{1/2} \| |R_{r}|^{1/2} E P_r \|_{2}\delta_r^\prime )
  \end{align}
  with $ C_\epsilon = 17/ (4 \epsilon)^2 $.
\end{corollary}

\begin{proof}
     Corollary \ref{cor_EigenvaluePerturbationBound} follows from Theorem \ref{thm_EigenvaluePerturbationSeries} in combination with the Hoffman-Wielandt inequality. First, note that the finite family $(\widehat\lambda_j-\mu_r)_{j\in I_r}$ coincides with $(\lambda_k(\widehat P_r(\widehat\Sigma-\mu_{r} I)\widehat P_r))_{k=1}^{m_r}$. 
  Combining this with the Hoffman-Wielandt inequality, we get
  \begin{align}
    \Big( \sum_{j \in I_r} (\widehat{\lambda}_{j} - \mu_{r} )^{2} \Big)^{1/2} 
    &= 
    \Big(\sum_{k = 1}^{m_r} (\lambda_k(\widehat P_r(\widehat\Sigma - \mu_{r} I)\widehat P_r))^2\Big)^{ 1 / 2 }
    \\
    & \le
    \Big( \sum_{k = 1}^{m_r} (\lambda_k(\widehat P_r(\widehat\Sigma-\mu_{r} I)\widehat P_r)-\lambda_k(P_rEP_r))^2\Big)^{ 1 / 2 } \notag \\
    &+\Big(\sum_{k=1}^{m_r}(\lambda_k(P_rEP_r))^2\Big)^{ 1 / 2 }
    \notag 
    \\
    & \le 
    \|\widehat P_r(\widehat\Sigma-\mu_{r} I)\widehat P_r)-P_rEP_r\|_2+\|P_rEP_r\|_2.
    \notag 
  \end{align}
  Note that we can apply the finite-dimensional version of Hofffman-Wielandt (see e.g.~\cite{MR2906465}), since both $\widehat P_r$ and $P_r$ have an $m_r$-dimensional range.
  The claim now follows from Theorem \ref{thm_EigenvaluePerturbationSeries} with $p=1$ using that $ Q_r^{(1)} = P_r E P_r $.
\end{proof}

Corollary \ref{cor_EigenvaluePerturbationBound} only includes (upper bounds) for the Hilbert-Schmidt norm of the two leading Taylor coefficients $ Q_r^{(1)}$ and $ Q_r^{(2)}$, and is thus expected to be close to optimal (see \cite{Wahl2026PertubationExpansionsI} for a more detailed discussion).

\begin{corollary}[Projector perturbation bound]
  \label{cor_ProjectorPerturbationBound}
  If \( \delta_r' < 1 / 4 - \varepsilon \) for some $\epsilon\in(0,1/4)$, then
  \begin{align}
    \| \widehat{P}_{r} - P_r\|_{2} 
    & \le 
     \sqrt{2}\|  R_{r} E P_{r} \|_{2}+C_\epsilon g_r^{-1/2}\| | R_{r} |^{1/2} E P_{r} \|_{2}\delta_r^\prime
  \end{align}
  with $C_\epsilon=16/(4\epsilon)^2$.
\end{corollary}

\noindent Corollary \ref{cor_ProjectorPerturbationBound} follows from Theorem
\ref{thm_ProjectorPerturbationSeries} in the case that $p = 2$.



\section{Proofs of the spectral perturbation bounds}
\label{sec_ProofsOfThePerturbationBounds}

\subsection{Eigenvalue separation and Taylor remainder}
\label{ssec_EigenvalueSeparationAndTaylorRemainder}

When \( \delta_{r} < 1 \), we obtain a guarantee that the perturbed eigenvalues \( \widehat{ \lambda }_{j} \) stay close to their unperturbed counterparts.
In particular, for \( \delta_{r} < 1/2 \), the groups corresponding to \( \mu_{r-1}, \mu_{r} \) and \( \mu_{r + 1} \) are contained in separate intervals that do not overlap.

\begin{lemma}[Eigenvalue separation]
  \label{lem_EigenvalueSeparation}
  If \( \delta_{r} < 1 \), then we have 
  \begin{align}
            | \widehat{ \lambda }_{j} - \mu_{r} | 
    & \le
            \delta_{r} g_{r} 
            \qquad \qquad \qquad
            \qquad \text{for all } j \in I_{r},
    \\
            \widehat{ \lambda }_{k} - \mu_{r + 1} 
    & \le 
            \delta_{r} ( \mu_{r} - \mu_{r + 1} ) 
            \quad \qquad \ \ \text{for all } k \in I_{r + 1},
    \notag
    \\
            \widehat{ \lambda }_{k} - \mu_{r - 1} 
    & \ge
            - \delta_{r} ( \mu_{r - 1} - \mu_{r} ) 
            \qquad \ \ \; \text{for all } k \in I_{r - 1}.
    \notag
  \end{align}
\end{lemma}

\begin{proof} 
  Set
  \begin{align}
        T_{ \ge r }:
    =
        \sum_{ s \ge r }
        \frac{1}{ \sqrt{ \mu_{r} + \delta_{r} g_{r} - \mu_{s} } } 
        P_{s},
    \qquad 
        T_{ \le r }: 
    = 
        \sum_{ s \le r } 
        \frac{1}{ \sqrt{ \mu_{s} + \delta_{r} g_{r} - \mu_{r} } } 
        P_{s},
  \end{align}
  where the sum is taken over the indices of the distinct eigenvalues \( (
  \mu_{r} )_{ r \ge 1 } \) only.
  We have
  \begin{align}
            \| T_{ \ge r } E T_{ \ge r } \|_{ \infty } 
    & \le 
            \| ( | R_{r} |^{1/2} + ( \delta_{r} g_{r} )^{-1/2} P_{r} ) E
               ( | R_{r} |^{1/2} + ( \delta_{r} g_{r} )^{-1/2} P_{r} )
            \|_{ \infty } 
    \\
    & \le 
            \delta_{r}^{-1} 
            \| ( | R_{r} |^{1/2} + g_{r}^{-1/2} P_{r} ) E
               ( | R_{r} |^{1/2} + g_{r}^{-1/2} P_{r} )
            \|_{ \infty } 
    \le 
            1.
    \notag
  \end{align}
  For any \( j \in I_{r} \), this implies
  \begin{align}
    & \ \ \ \
          \Big\| \sum_{ k, l \ge j }
                 \frac{1}
                 { \sqrt{ \lambda_{j} + \delta_{r} g_{r} - \lambda_{k} } } 
                 \frac{1}
                 { \sqrt{ \lambda_{j} + \delta_{r} g_{r} - \lambda_{l} } } 
                 ( u_{k} \otimes u_{k} ) E ( u_{l} \otimes u_{l} ) 
          \Big\|_{ \infty }
    \\
    & \le 
          \| T_{ \ge r } E T_{ \ge r } \|_{ \infty } 
    \le 
          1.
    \notag 
  \end{align}
  Proposition 1 from \cite{JirakWahl2018PerturbationBoundsUnderRelativeGap}
  then yields \( \widehat{ \lambda }_{j} - \lambda_{j} \le \delta_{r} g_{r} \).  
  Analogously, we obtain 
  \( \widehat{ \lambda }_{j} - \lambda_{j} \ge - \delta_{r} g_{r} \), which
  completes the proof of the first claim.
  By considering 
  \begin{align}
        T_{ > r }:
    =
        \sum_{ s > r }
        \frac{1}{ \sqrt{ \mu_{r} + \delta_{r} g_{r} - \mu_{s} } } 
        P_{s},
    \qquad 
        T_{ <  r }: 
    = 
        \sum_{ s < r } 
        \frac{1}{ \sqrt{ \mu_{s} + \delta_{r} g_{r} - \mu_{r} } } 
        P_{s},
  \end{align}
  the second and third claim can be verified analogously.
\end{proof}

\noindent In the situation of Lemma \ref{lem_EigenvalueSeparation}, we obtain
an explicit Taylor remainder.

\begin{lemma}[Taylor remainder]
  \label{lem_TaylorRemainder}
  Suppose that \( \delta_{r} < 1 \).
  Then, for any integer \( p \ge 1 \), 
  \begin{align}
    \label{eq_A_TaylorRemainder}
        \widehat{P}_{r} - \sum_{n = 0}^{p - 1} P_{r}^{ (n) }
    & = 
        (-1)^{p - 1} 
        \sum_{k_{1}, \dots k_{p}, \ge 0} 
        R_{r}^{ ( k_{1} ) } E \dots E R_{r}^{ ( k_{p} ) } 
        E \widehat{R}_{r}^{ ( p - k_{1} - \dots - k_{p} ) }
  \end{align}
  with 
  \( 
    \widehat{R}_{r}^{ (k) } = - \sum_{ j \in I_{r} }  
                                ( \widehat{ \lambda }_{j} - \mu_{r} )^{ - k }
                                \widehat{u}_{j} \otimes \widehat{u}_{j} 
  \),
  for \( k \le 0 \) and \( \widehat{R}_{r}^{ (k) } = \widehat{R}_{r} ^{k} \),
  for  \( k > 0 \), where 
  \begin{align}
    \widehat{R}_{r}: = \sum_{ s \ne r } \sum_{ j \in I_{s} } 
                       ( \widehat{ \lambda }_{j} - \mu_{r}  )^{-1} 
                       \widehat{u}_{j} \otimes \widehat{u}_{j}.
  \end{align}
  is the empirical reduced resolvent.
  It is well defined, since \( \widehat{ \lambda }_{j} - \mu_{r} \ne 0 \) for 
  \( j \in I_{s}, s \ne r \)  due to Lemma 
  \normalfont \ref{lem_EigenvalueSeparation}.
\end{lemma}

\begin{proof} 
  We prove the statement in Equation \eqref{eq_A_TaylorRemainder} via induction
  over \( p \in \mathbb{N} \).
  For \( p = 1 \), we write \( \widehat{P}_{r} - P_r \) as
  \begin{align}
        \widehat{P}_{r} - P_{r} 
    & = 
        ( I - P_{r} ) \widehat{P}_{r} - P_{r} ( I - \widehat{P}_{r} ).
  \end{align}
  Using \begin{align}
  \label{eq_ProjectureToEigenvalueEquality}
  ( u_{j} \otimes u_{j} ) E ( \widehat{u}_{k} \otimes \widehat{u}_{k} ) 
  & = 
  ( \widehat{ \lambda }_{k} - \lambda_{j} ) 
  ( u_{j} \otimes u_{j} ) ( \widehat{u}_{k} \otimes \widehat{u}_{k} ) 
\end{align}
  and Lemma \ref{lem_EigenvalueSeparation}, the second term on the right-hand is given by
  \begin{align}
      P_{r} ( I - \widehat{P}_{r} ) 
    = 
      \sum_{ j \in I_{r} } 
      \sum_{ s \ne r } \sum_{ k \in I_{s} }
      ( \widehat{ \lambda }_{k} - \mu_{r} )^{-1}
      ( u_{j} \otimes u_{j} ) E ( \widehat{u}_{k} \otimes \widehat{u}_{k} ) 
    = 
      P_{r} E \widehat{R}_{r}.
  \end{align}
  Further, with analogous reasoning, the first term on the right-hand side is 
  given by
  \begin{align}
      ( I - P_{r} ) \widehat{P}_{r} 
    =
      \sum_{ k \in I_{r} } 
      \sum_{ s \ne r } ( \widehat{ \lambda }_{k} - \mu_{s} )^{-1} P_{s}
      E ( \widehat{u}_{k} \otimes \widehat{u}_{k} ) 
    = 
      - \sum_{ k \in I_{r} } \sum_{ l = 1 }^{ \infty } 
      ( \widehat{ \lambda }_{k} - \mu_{r} )^{ l - 1 } 
      R_{r}^{l} E ( \widehat{u}_{k} \otimes \widehat{u}_{k} ). 
    \notag
  \end{align}
  For the last equality, we have used a geometric series expansion of
  \( 
      ( \widehat{ \lambda }_{k} - \mu_{s} )^{-1} 
    = 
      - ( \mu_{s} - \mu_{r} )^{-1} / 
        ( 1 - ( \widehat{ \lambda }_{k} - \mu_{r} ) / ( \mu_{s} - \mu_{r} ) )
  \),
  which is well defined due to Lemma \ref{lem_EigenvalueSeparation}.
  It follows that 
  \begin{align}
            \widehat{P}_{r} - P_{r} 
    & = 
            \sum_{ l = 1 }^{ \infty } 
            R_{r}^{l} E \widehat{R}_{r}^{ ( 1 - l ) }
          + 
            R_{r}^{ (0) } E \widehat{R}_{r}^{ (1) }
    = 
            \sum_{ l = 0 }^{ \infty } 
            R_{r}^{l} E \widehat{R}_{r}^{ ( 1 - l ) },
    \notag
  \end{align}
  which establishes the induction hypothesis.

  For the induction step, we assume that the statement is true for a fixed 
  \( p \) and consider the fact that the induction hypothesis can be written as
  \begin{align}
          \widehat{R}_{r}^{ (0) } 
    & =
          R_{r}^{ (0) } 
        - 
          \sum_{l = 0}^{ \infty } 
          R_{r}^{ (l) } E \widehat{R}_{r}^{ (1 - l) }.
  \end{align}
  We will show that analogously
  \begin{align}
    \label{eq_RKSmallerZero}
          \widehat{R}_{r}^{ (k) } 
    & = 
        \qquad \
        - \sum_{l = 0}^{ \infty } 
          R_{r}^{ (l) } E \widehat{R}_{r}^{ (k + 1 - l) },
        \qquad k < 0;
    \\
    \label{eq_RKLargerZero}
          \widehat{R}_{r}^{ (k) } 
    & = 
          R_{r}^{ (k) } 
        - 
          \sum_{l = 0}^{ \infty } 
          R_{r}^{ (l) } E \widehat{R}_{r}^{ (k + 1 - l) }, 
        \qquad k \ge 0.
  \end{align}
  Setting \( k: = p - k_{1} - \dots - k_{p} \), it then follows from Equations
  \eqref{eq_RKSmallerZero}, \eqref{eq_RKLargerZero} and the induction hypothesis
  that 
  \begin{align}
         \widehat{P}_{r} - \sum_{n = 0}^{p - 1} P_{r}^{ (n) } 
    & =  
         (-1)^{p - 1} \sum_{k < 0} 
         \sum_{ \substack{ k_{1}, \dots, k_{p} \ge 0 \\
                           k_{1} + \dots + k_{p} = p - k 
                         }
              }
         R_{r}^{ ( k_{1} ) } E \dots E R_{r}^{ ( k_{p} ) } E
         \Big( - \sum_{l = 0}^{ \infty } 
                 R_{r}^{ (l) } E \widehat{R}_{r}^{ (k + 1 - l) } 
         \Big)
    \notag
    \\
    & + 
            (-1)^{p - 1} \sum_{k \ge 0} 
            \sum_{ \substack{ k_{1}, \dots, k_{p} \ge 0 \\
                              k_{1} + \dots + k_{p} = p - k 
                            }
              }
            R_{r}^{ ( k_{1} ) } E \dots E R_{r}^{ ( k_{p} ) } E
            \Big( 
              R_{r}^{ (k) } 
            - 
              \sum_{l = 0}^{ \infty } 
              R_{r}^{ (l) } E \widehat{R}_{r}^{ ( k + 1 - l ) } 
            \Big)
    \notag
    \\
    & = 
            (-1)^{p} 
            \sum_{k_{1}, \dots, k_{p + 1} \ge 0} 
            R_{r}^{ ( k_{1} ) } E \dots E R_{r}^{ ( k_{p + 1} ) } E 
            \widehat{R}_{r}^{ (p + 1 - k_{1} - \dots - k_{p + 1}) } 
            +
            P_{r}^{ (p) },
    \notag
  \end{align}
  which yields the statement in Lemma \ref{lem_TaylorRemainder} by rearranging
  the terms.

  It remains to be shown that Equation \eqref{eq_RKSmallerZero} and
  \eqref{eq_RKLargerZero} are true for \( k < 0 \) and \( k > 0 \) respectively. 
  For \( k < 0 \), we have that
  \begin{align}
          \widehat{R}_{r}^{ (k) } 
    & = 
        - \sum_{ j \in I_{r} } 
          ( \widehat{ \lambda }_{j} - \mu_{r} )^{-k} 
          P_{r} ( \widehat{u}_{j} \otimes \widehat{u}_{j} )
        - \sum_{ j \in I_{r} } 
          ( \widehat{ \lambda }_{j} - \mu_{r} )^{-k} 
          ( I - P_{r} ) ( \widehat{u}_{j} \otimes \widehat{u}_{j} ),
    \notag 
  \end{align}
  where, using Equation \eqref{eq_ProjectureToEigenvalueEquality}, the first
  term on the right-hand side is equal to
  \begin{align}
      - \sum_{ j \in I_{r} } 
        ( \widehat{ \lambda }_{j} - \mu_{r} )^{ - k - 1 } 
        P_{r} E ( \widehat{u}_{j} \otimes \widehat{u}_{j} ) 
    = 
      - R_{r}^{ (0) } E \widehat{R}_{r}^{ ( k + 1 ) }.
  \end{align}
  Reusing the geometric series expansion from before, the second term on the
  right-hand side can be written as
  \begin{align}
    & \ \ \ \
        - \sum_{s \ne r} \sum_{ j \in I_{r} } \sum_{ j' \in I_{s} } 
          ( \widehat{ \lambda }_{j} - \mu_{r} )^{-k} 
          ( \widehat{ \lambda }_{j} - \mu_{s} )^{-1}
          ( u_{j'} \otimes u_{j'} ) E ( \widehat{u}_{j} \otimes \widehat{u}_{j} ) 
    \\
    & = 
        - \sum_{ j \in I_{r} } ( \widehat{ \lambda }_{j} - \mu_{r} )^{-k} 
          \sum_{ s \ne r } 
          ( \widehat{ \lambda }_{j} - \mu_{s} )^{-1}
          P_{s} E ( \widehat{u}_{j} \otimes \widehat{u}_{j} )
    \notag
    \\
    & =
          \sum_{ j \in I_{r} } \sum_{l = 1}^{ \infty } 
          ( \widehat{ \lambda }_{j} - \mu_{r} )^{-k - 1 + l} 
          R_{r}^{ (l) } E ( \widehat{u}_{j} \otimes \widehat{u}_{j} ) 
    = 
        - \sum_{l = 1}^{ \infty } 
          R_{r}^{ (l) } E \widehat{R}_{r}^{ (k + 1 - l) }.
    \notag
  \end{align}
  Together, this yields Equation \eqref{eq_RKSmallerZero}.

  For \( k > 0 \), we write
  \begin{align}
    \label{eq_RKLargerZeroExpanded_1}
        \widehat{R}_{r}^{ (k) } 
    & = 
        ( I - P_{r} ) \widehat{R}_{r}^{k} + P_{r} \widehat{R}_{r}^{k}.
  \end{align}
  Focusing on the first term on the right-hand side, we note that inserting \eqref{eq_ProjectureToEigenvalueEquality} into \( R_{r} E \widehat{R}_{r} \) yields \( ( I - P_{r} ) \widehat{R}_{r} = R_{r} ( I - \widehat{P}_{r} ) - R_{r} E \widehat{R}_{r} \)
  and inductively
  \begin{align}
    \label{eq_RKLargerZeroExpanded_2}
            ( I - P_{r} ) \widehat{R}_{r}^{k} 
    & = 
            R_{r}^{k} ( I - \widehat{P}_{r} ) 
          - 
            \sum_{ l = 1 }^{k} 
            R_{r}^{l} E \widehat{R}_{r}^{k + 1 - l}.
  \end{align}
  Together, Equations \eqref{eq_RKLargerZeroExpanded_1} and
                      \eqref{eq_RKLargerZeroExpanded_2} then amount to
  \begin{align}
            \widehat{R}_{r}^{k} 
    & = 
            R_{r}^{k} ( I - \widehat{P}_{r} ) 
          - 
            \sum_{ l = 1 }^{k} 
            R_{r}^{l} E \widehat{R}_{r}^{k + 1 - l}
          + 
            P_{r} \widehat{R}_{r}^{k}
    \\
    & = 
            R_{r}^{k} 
          - 
            \sum_{ l = k + 1 }^{ \infty } 
            R_{r}^{l} E \widehat{R}_{r}^{(l - 1 - k)} 
          - 
            \sum_{ l = 1 }^{k} 
            R_{r}^{l} E \widehat{R}_{r}^{k + 1 - l}
          + 
            P_{r} \widehat{R}_{r}^{k}
    \notag
    \\
    & = 
    \notag
        R_{r}^{k} 
      - 
        \sum_{ l = 1 }^{ \infty } 
        R_{r}^{l} E \widehat{R}_{r}^{ (l - 1 - k) } 
      + 
        P_{r} E \widehat{R}_{r}^{k + 1}
    = 
        R_{r}^{k} 
      - 
        \sum_{l = 0}^{ \infty } 
        R_{r}^{l} E \widehat{R}_{r}^{ (l - 1 - k) },
  \end{align}
  where again, we have used Equation \eqref{eq_ProjectureToEigenvalueEquality}
  and the geometric series expansion to obtain
  \begin{align}
        - R_{r}^{k} \widehat{P}_{r} 
    & = 
        - R_{r}^{k} ( I - P_{r} ) \widehat{P}_{r} 
    = 
        - R_{r}^{k} 
          \sum_{ s \ne r } 
          \sum_{ j  \in I_{s} }
          \sum_{ j' \in I_{r} } 
          ( \widehat{ \lambda }_{j'} - \lambda_{j} )^{-1} 
          ( u_{j} \otimes u_{j} ) E 
          ( \widehat{u}_{j'} \otimes \widehat{u}_{j'}) 
    \\
    & = 
          R_{r}^{k}
          \sum_{ j' \in I_{r} } 
          \sum_{ l = 1 }^{ \infty } 
          ( \widehat{ \lambda }_{ j' } - \mu_{r} )^{ l - 1 } R_{r}^{l}
          E ( \widehat{u}_{ j' } \otimes \widehat{u}_{ j' }) 
    = 
        - \sum_{ l = k + 1 }^{ \infty } 
          R_{r}^{l} E \widehat{R}_{r}^{ ( l - 1 - k ) }.
    \notag
  \end{align}
  This finishes the proof.
\end{proof}


\subsection{Norm estimates and contraction inequalities}
\label{ssec_NormEstimatesAndContractionInequalities}

In order to bound the Taylor remainder, it will be necessary to enumerate and estimate the summands in the terms \( P_{r}^{ ( n ) } \).

\begin{lemma}
  \label{lem_NumberOfTuples}
  The number of \( ( n + 1 ) \)-tuples 
  \( ( k_{1},  \dots,  k_{n + 1} ) \in \mathbb{N}_{0}^{ n + 1 } \) such that
  \( k_{1} + \dots + k_{ n + 1 } = m \) is equal to 
  \( \binom{n + m}{n} \le 2^{ n + m } \). Moreover, the number \( ( n + 1 ) \)-tuples 
  \( ( k_{1},  \dots,  k_{n + 1} ) \in \mathbb{N}_{0}^{ n + 1 } \) such that \(k_1\geq 1\) and
  \( k_{1} + \dots + k_{ n + 1 } = m \) is equal to 
  \( \binom{n + m -1}{n} \le 2^{ n + m - 1} \).  
\end{lemma}

\begin{lemma}[Elementary norm estimates]
  \label{lem_ElementaryNormEstimates}
  For \( n \ge 1 \) and tuple \( (k_{1}, \dots, k_{n + 1}) \in \mathbb{N}_{0}^{n + 1} \) with \( k_{1} + \dots + k_{n + 1} = m \), it holds that
  \begin{enumerate}
    \item[(i)]
      \( 
          \| R_{r}^{ (k_{1}) } E \cdots E R_{r}^{ (k_{n + 1}) } \|_{ \infty } 
      \le g_{r}^{n - m} \delta_{r}'^{n}.
      \) 
  \end{enumerate}
  Moreover, if \( k_{a} = 0 \) for some \( 1 \le a \le n + 1 \) and \( m \ge 1 \), then 
  \begin{enumerate}[label=(\roman*)]
    \item[(ii)]
      \( 
            \| R_{r}^{ ( k_{1} ) } E \dots E R_{r}^{ ( k_{n + 1} ) } \|_{2}
        \le g_{r}^{ n - m - 1 / 2 } \| | R_{r} |^{1/2} E P_{r} \|_{2}
            \delta_{r}'^{n - 1}
      \)
  \end{enumerate}
  Analogous results hold for \( R_{r}^{ ( k_{1} ) } \) or \( R_{r}^{ ( k_{ n + 1 }) } \) replaced by \( | R_{r} |^{1/2} \) and setting \( k_{1}, k_{n + 1} = 1/2 \).
\end{lemma}

\begin{proof}
  For the statement in (i), we focus on the case where \( k_{1}, k_{ n + 1 } \ge 1 \). 
  The other cases follow analogously.
  Set
  \begin{align}
    S_{r}^{ (k) }: = \begin{cases} 
                       P_{r},           & k = 0, \\
                       | R_{r} |^{1/2}, & k > 0.
                     \end{cases}
  \end{align}
  Then,
  \begin{align}
    \label{eq_C_ProductDecomposition}
          \| R_{r}^{ (k_{1}) } E \dots E R_{r}^{ (k_{ n + 1 }) } \|_{\infty} 
    & \le 
          g_{r}^{-\sum_{a = 1}^{n + 1} (k_{a} - 1)^{+} - 1} 
          \prod_{ a = 1 }^{n}
          \| S_{r}^{ (k_{a}) } E S_{r}^{ (k_{a + 1}) } \|_{\infty},
  \end{align}
  where \( \sum_{a = 1}^{n + 1} (k_{a} - 1)^{+} \) is the sum of the values of
  the \( ( k_{a} )_{a} \) in excess of one.
  The factors in the last product are of the form
  \begin{align}
    \label{eq_C_ProductDecompositionTerms}
    \| | R_{r} |^{1/2} E | R_{r} |^{1/2} \|_{\infty}, \qquad 
    \| | R_{r} |^{1/2} E P_{r}           \|_{\infty} \qquad \text{or} \qquad 
    \| P_{r}           E P_{r}           \|_{\infty}
  \end{align}
  and the number of values \( k_{a} \) equal to zero is given
  \begin{align}
    | \{ a: k_{a} = 0 \} | & = n + 1 - m 
                             + \sum_{a = 1}^{n + 1} (k_{a} - 1)^{+}, 
  \end{align}
  where \( n + 1 - m \) is the minimal number of zeros possible and the sum is
  the number of zeros for the particular instance of the \( ( k_{a} )_{a} \)
  exceeding the minimum.
  Redistributing, \( | \{ a: k_{a} = 0 \} | \) many factors \( g_{r}^{-1} \)
  onto the terms in Equation \eqref{eq_C_ProductDecompositionTerms} yields
  \begin{align}
        \| R_{r}^{ ( k_{1} ) } E \dots E R_{r}^{ ( k_{n + 1} ) } \|_{\infty} 
    \le 
        g_{r}^{n - m} \delta_{r}'^{n}.
  \end{align}

  For the statement in (ii), we note that for \( k_{a} = 0 \), there has to
  exist a term \( | R_{r} |^{1/2} E P_{r} \) in the product from Equation
  \eqref{eq_C_ProductDecomposition}.
  By consecutively applying the inequality 
  \( \| A B \|_{2} \le \| A \| _{\infty} \| B \|_{2} \), this yields (ii).
\end{proof}

\noindent Finally, we need to quantify the contractive properties of the
operators \( \widehat{R}^{ ( k ) } \), \( k \in \mathbb{Z} \).

\begin{lemma}[Contraction]
  \label{lem_Contraction}
  If \( \delta_{r} < 1/2 \), then
  \begin{enumerate}[label=(\roman*)]
    \item \(    \| |R_{r}|^{-1/2} \widehat{P}_{r} \|_{2}
            \le \| |R_{r}|^{1/2} E P_{r} \|_{2} / (1 - 2 \delta_{r})
          \);

    \item \(    \| P_{r} \widehat{R}_{r}^{k} \|_{2} 
            \le (1 - \delta_{r})^{-k} g_{r}^{- k - 1/2}  
                \| | R_{r} |^{1/2} E P_{r} \|_{2} / (1 - 2 \delta_{r})
          \) for any \( k \ge 1 \);

    \item \(    \| | R_{r} |^{-1/2} \widehat{R}_{r}^{k}  \|_{ \infty } 
            \le (1 - \delta_{r})^{-k - 1} g_{r}^{-k + 1/2} 
          \) for any \( k \ge 1 \).
  \end{enumerate}
\end{lemma}

\begin{proof}
  For (i), Equation \eqref{eq_ProjectureToEigenvalueEquality} yields that for
  \( s \ne r \),
  \begin{align}
          \| P_{s} \widehat{P}_{r} \|_{2}^{2} 
    & = 
          \sum_{ k \in I_{s} } \sum_{ j \in I_{r} } 
          \| ( u_{k} \otimes u_{k} ) 
             ( \widehat{u}_{j} \otimes \widehat{u}_{j}) 
          \|_{2}^{2}
    = 
          \sum_{ k \in I_{s} } \sum_{ j \in I_{r} } 
          \frac{ \| ( u_{k} \otimes u_{k} ) E
                    ( \widehat{u}_{j} \otimes \widehat{u}_{j}) 
                 \|_{2}^{2}
               }{ ( \widehat{ \lambda }_{j} - \lambda_{k} )^{2} } 
    \\
    & \le 
          \frac{1}{ ( 1 - \delta_{r} )^{2} }
          \sum_{ k \in I_{s} } \sum_{ j \in I_{r} } 
          \frac{ \| ( u_{k} \otimes u_{k} ) E
                    ( \widehat{u}_{j} \otimes \widehat{u}_{j}) 
                 \|_{2}^{2}
               }{ ( \lambda_{j} - \lambda_{k} )^{2} } 
    = 
          \frac{1}{ ( 1 - \delta_{r} )^{2} }
          \frac{ \| P_{s} E \widehat{P}_{r} \|_{2}^{2} }
               { ( \mu_{s} - \mu_{r})^{2} }, 
    \notag
  \end{align}
  where Lemma \ref{lem_EigenvalueSeparation} implies the first inequality and
  also guarantees that all denominators are well defined.
  Consequently, 
  \begin{align}
    \label{eq_Contraction_i_1}
        \| | R_{r} |^{-1/2} \widehat{P}_{r} \|_{2} 
    & = 
        \| \sum_{ s \ne r } | \mu_{s} - \mu_{r} |^{1/2} P_{s} 
           \widehat{P}_{r} 
        \|_{2}
    \le 
        \frac{1}{ 1 - \delta_{r} } 
        \| | R_{r} |^{1/2} E \widehat{P}_{r} \|_{2}. 
  \end{align}
  Using \( I = P_{r} + I - P_{r} \) and 
  \( I - P_{r} = | R_{r} |^{1/2} | R_{r} |^{-1/2} \), we have
  \begin{align}
            \| | R_{r} |^{1/2} E \widehat{P}_{r} \|_{2}
    & \le
            \| | R_{r} |^{1/2} E       P_{r}   \widehat{P}_{r} \|_{2}
          + 
            \| | R_{r} |^{1/2} E ( I - P_{r} ) \widehat{P}_{r} \|_{2}
    \\
    & \le 
            \| | R_{r} |^{1/2} E P_{r} \|_{2}
          + 
            \| | R_{r} |^{1/2} E | R_{r} |^{1/2} \|_{\infty}
            \| | R_{r} |^{-1/2} \widehat{P}_{r} \|_{2}
    \notag
    \\
    & \le 
            \| | R_{r} |^{1/2} E P_{r} \|_{2}
          + 
            \delta_{r}
            \| | R_{r} |^{-1/2} \widehat{P}_{r} \|_{2}
    \notag
  \end{align}
  Inserting this into Equation \eqref{eq_Contraction_i_1}, yields
  \begin{align}
            \| | R_{r} |^{-1/2} \widehat{P}_{r} \|_{2} 
    & \le 
            \frac{ \| | R_{r} |^{-1/2} E P_{r} \|_{2} }
                 { 1 - \delta_{r}                       } 
          + 
            \frac{ \delta_{r} \| | R_{r} |^{-1/2} \widehat{P}_{r} \|_{2} }
                 { 1 - \delta_{r}                                           },
  \end{align}
  which gives the result after rearranging.

  For (ii), by Lemma \ref{lem_EigenvalueSeparation},
  \begin{align}
          \| P_{r} \widehat{R}_{r}^{k} \|_{2} 
    & = 
          \sqrt{ \sum_{ s \ne r } \sum_{ j \in I_{s} } 
                 ( \widehat{ \lambda }_{j} - \mu_{r} )^{-2 k} 
                 \| P_{r} \widehat{u}_{j} \otimes \widehat{u}_{j} \|_{2}^{2}
          } 
    \\
    & \le 
          ( 1 - \delta_{r} )^{-k} g_{r}^{-k} 
          \| ( I - P_{r} ) \widehat{P}_{r} \|_{2}
    \le 
          ( 1 - \delta_{r} )^{ - k } g_{r}^{-k - 1/2} 
          \| | R_{r} |^{-1/2} \widehat{P}_{r} \|_{2}.
    \notag
  \end{align}
  By (i), we then obtain
  \begin{align}
           \| P_{r} \widehat{R}_{r}^{k} \|_{2} 
    & \le 
           g_{r}^{- k - 1/2} 
           \| | R_{r} |^{1/2} E P_{r} \|_{2} 
           \frac{ ( 1 - \delta_{r} )^{-k} }{ 1 - 2 \delta_{r} }.
  \end{align}

  Finally, for (iii), by Equation \eqref{eq_RKLargerZeroExpanded_2} with 
  \( k = 1 \),
  \begin{align}
            \| | R_{r} |^{-1/2} \widehat{R}_{r} \|_{\infty}
    & = 
            \| | R_{r} |^{-1/2} 
               ( R_{r} ( I - \widehat{P}_{r} ) - R_{r} E \widehat{R}_{r} )
            \|_{\infty} 
    \\
    & \le
            \| | R_{r} |^{1/2} ( I - \widehat{P}_{r} ) \|_{\infty} 
          + 
            \| | R_{r} |^{1/2} E ( P_{r} + I - P_{r} ) \widehat{R}_{r}
            \|_{\infty} 
    \notag
    \\
    & \le 
            g_{r}^{-1/2} 
          + 
            \| | R_{r} |^{1/2} E P_{r} \widehat{R}_{r} \|_{\infty}
          + 
            \| | R_{r} |^{1/2} E | R_{r} |^{1/2} | R_{r} |^{-1/2} \widehat{R}_{r} 
            \|_{\infty}
    \notag
    \\
    & \le 
            g_{r}^{-1/2} 
          + 
            \| | R_{r} |^{1/2} E P_{r} \|_{\infty}
            g_{r}^{-1} ( 1 - \delta_{r} )^{-1} 
          + 
            \delta_{r}' \| | R_{r} |^{-1/2} \widehat{R}_{r} \|_{\infty}, 
    \notag
    \\
    & \le 
            g_{r}^{-1/2} 
          + 
            g_{r}^{-1/2} 
            \frac{ \delta_{r}' }{ 1 - \delta_{r} }
          + 
            \delta_{r}' \| | R_{r} |^{-1/2} \widehat{R}_{r} \|_{\infty},
    \notag
  \end{align}
  where we have used Lemma \ref{lem_EigenvalueSeparation} for the third inequality.
  Consequently,
  \begin{align}
          \| | R_{r} |^{-1/2} \widehat{R}_{r} \|_{\infty}
    & \le 
          \frac{ g_{r}^{-1/2} }{ ( 1 - \delta_{r} )^{2} } 
  \end{align}
  and for \( k > 1 \), by Lemma \ref{lem_EigenvalueSeparation},
  \begin{align}
           \| | R_{r} |^{-1/2} \widehat{R}_{r}^{k} \|_{ \infty } 
    & \le 
           \| | R_{r} |^{-1/2} \widehat{R}_{r} \|_{ \infty } 
           \| \widehat{R}_{r}^{ k - 1 } \|_{ \infty }
    \\
    & \le 
           \frac{ g_{r}^{-1/2} ( ( 1 - \delta_{r} ) g_{r} )^{ -(k - 1) } }
                { ( 1 - \delta_{r} )^{2}               } 
    = 
           g_{r}^{ - k + 1/2 } ( 1 - \delta_{r} )^{ -(k + 1) },
    \notag
  \end{align}
  which finishes the proof.
\end{proof}


\subsection{Derivation of the perturbation series}
\label{ssec_DerivationOfThePerturbationSeries}

The set of Lemmas from the last section, now allows to prove the result for the perturbation series of the projectors.

\begin{proof}[\normalfont \textbf{Proof of Theorem \ref{thm_ProjectorPerturbationSeries}} (Projector perturbation series)]
  \label{prf_ProjectorPerturbationSeries} 
  The absolute convergence follows by applying Lemma \ref{lem_ElementaryNormEstimates} (ii) directly.
  From Lemma \ref{lem_TaylorRemainder}, it now follows that 
  \begin{align}
    \label{eq_TaylorRemainderEstimate}
          \| \widehat{P}_{r} - \sum_{ n = 0 }^{ p - 1 }  P_{r}^{ (n) } \|_{2} 
    & \le 
          \sum_{ k \in \mathbb{Z} } 
          \sum_{ \substack{ k_{1}, \dots, k_{p} \ge 0 \\
                            k_{1} + \dots + k_{p} = p - k 
                          }
               } 
          \| R_{r}^{ ( k_{1} ) } E \cdots E R_{r}^{ ( k_{p} ) } 
             E \widehat{R}_{r}^{ (k) } 
          \|_{2}.
  \end{align}
  For the summands on the right-hand side of Equation
  \eqref{eq_TaylorRemainderEstimate}, we distinguish the cases \( k \le 0 \)
  and \( k \ge 1 \). 

  First, consider the case $k\leq 0$. By Lemma \ref{lem_EigenvalueSeparation},
  we have
  \begin{align}
    & \ \ \ \
         \| R_{r}^{ ( k_{1} ) } E \cdots E R_{r}^{ ( k_{p} ) } 
                                         E \widehat{R}_{r}^{ (k) }
         \|_{2}^2 
    = 
         \sum_{ j \in I_{r} } 
         ( \widehat{ \lambda }_{j} - \mu_{r} )^{-2k}
         \| R_{r}^{ ( k_{1} ) } E \cdots E R_{r}^{ ( k_{p} ) } E 
                 ( \widehat{u}_{j} \otimes \widehat{u}_{j} )
         \|_{2}^2
   \\
   & \le
         \max_{ j \in I_{r}} | \widehat{ \lambda }_{j} - \mu_{r} |^{-2k}
         \| R_{r}^{ ( k_{1} ) } E \cdots E R_{r}^{ ( k_{p} ) } E \widehat{P}_{r}  
         \|_{2}^2
     \le 
         \delta_r^{-2k} g_r^{-2k}
         \| R_{r}^{ ( k_{1} ) } E \cdots E R_{r}^{ ( k_{p} ) } E \widehat{P}_{r}  
         \|_{2}^2.
  \end{align}
  Thus, by the identiy \( I - P_{r} = | R_{r} |^{1/2} | R_{r} |^{-1/2} \) and
  simple properties of the Hilbert-Schmidt norm, we get
  \begin{align}
    \label{eq_KLess0_1}
    & \ \ \ \
            \| R_{r}^{ ( k_{1} ) } E \cdots E R_{r}^{ ( k_{p} ) } 
                                            E \widehat{R}_{r}^{ (k) }
            \|_{2}
    \\
    & \le 
            \delta_r^{-k}g_r^{-k} 
            \| R_{r}^{ ( k_{1} ) } E \cdots E R_{r}^{ ( k_{p} ) } E 
               P_{r} \widehat{P}_{r}
            \|_{2} \notag
        + 
            \delta_r^{-k}g_r^{-k}\| R_{r}^{ ( k_{1} ) } E \cdots E R_{r}^{ ( k_{p} ) } E 
               (I - P_{r}) \widehat{P}_{r} 
            \|_{2}
    \notag
    \\
    & \le   
            \delta_r^{-k}g_r^{-k} 
            \| R_{r}^{ ( k_{1} ) } E \cdots E R_{r}^{ ( k_{p} ) } E P_{r}
            \|_{2}
         +  
            \delta_r^{-k} g_r^{-k} 
            \|   R_{r}^{ ( k_{1} ) } E \cdots E R_{r}^{ ( k_{p} ) } E 
               | R_{r} |^{1/2}
            \|_{\infty} 
            \| | R_{r} |^{-1/2} \widehat{P}_{r} \|_{2}.
    \notag
  \end{align}
  Using the norm estimates from Lemma \ref{lem_ElementaryNormEstimates}, we can
  now treat the last two terms separately.
  For the first term, Lemma \ref{lem_ElementaryNormEstimates}(ii) yields
  \begin{align}
    \label{eq_KLess0_2}
      \| R_{r}^{ ( k_{1} ) } E \cdots E R_{r}^{ ( k_{p} ) } E P_{r} \|_{2}
    \le 
          g_{r}^{k - 1/2} 
          \| | R_{r} |^{1/2} E P_{r} \|_{2} \delta_{r}'^{p - 1}. 
  \end{align}
  For the second term, we use Lemma \ref{lem_ElementaryNormEstimates}(i) together
  with the contraction result from Lemma \ref{lem_Contraction} (i) to obtain
  \begin{align}
    \label{eq_KLess0_3} 
          \| R_{r}^{ ( k_{1} ) } E \cdots E R_{r}^{ ( k_{p} ) } E 
             | R_{r} |^{1/2}
          \|_{\infty} 
          \|  | R_{r} |^{-1/2} \widehat{P}_{r} \|_{2}
    \le 
          g_{r}^{k - 1/2} 
          \| | R_{r} |^{1/2} E P_{r} \|_{2} 
          \frac{ \delta_{r}'^{p} }{ 1 - 2 \delta_{r} }.
  \end{align}
  Together, Equations \eqref{eq_KLess0_1},
  \eqref{eq_KLess0_2}, \eqref{eq_KLess0_3} and the inequality $\delta_r'\leq \delta_r$ yield
  \begin{align}
        \| R_{r}^{ ( k_{1} ) } E \cdots E R_{r}^{ ( k_{p} ) } 
           E \widehat{R}_{r}^{ (k) }
        \|_{2} 
    & \le 
        g_{r}^{ -1/2 } \| | R_{r} |^{1/2} E P_{r} \|_{2} 
        \Big( \delta_{r}'^{ p - 1 } \delta_{r}^{ - k } 
            + 
              \frac{ \delta_{r}'^{p} \delta_{r}^{-k} }{ 1 - 2 \delta_{r} }
        \Big) 
        \notag
    \\
    & \leq  
        g_{r}^{ -1/2 } \| | R_{r} |^{1/2} E P_{r} \|_{2} 
        \frac{ \delta_{r}'^{p - 1} \delta_{r}^{-k} }{ 1 - 2 \delta_{r} }.
    \notag
  \end{align}
  From there it follows that
  \begin{align}
    \label{eq_KLess0_4}
    & \ \ \ \
            \sum_{k \le 0} 
            \sum_{ \substack{ k_{1}, \dots, k_{p} \ge 0 \\
                              k_{1} + \dots + k_{p} = p - k 
                            }
                 } 
            \| R_{r}^{ ( k_{1} ) } E \dots E R_{r}^{ ( k_{p} ) } 
               E \widehat{R}_{r}^{ (k) } 
            \|_{2} 
    \\
    & \le 
            g_{r}^{-1/2} \| | R_{r} |^{1/2} E P_{r} \|_{2} 
            \frac{ \delta_{r}'^{p - 1}}{ 1 - 2 \delta_{r} }
            \sum_{k \le 0} 
            2^{2 p - 1 - k} \delta_{r}^{-k}
    \notag 
    \le 
            2 g_{r}^{-1/2} \| | R_{r} |^{1/2} E P_{r} \|_{2} 
            \frac{ ( 4 \delta_{r}' )^{p - 1} }{ ( 1 - 2 \delta_{r} )^{2} }
    \notag
  \end{align}
  by bounding the number of summands via Lemma \ref{lem_NumberOfTuples}.

  Next, consider the case $k\geq 1$. Then we have
  \begin{align}
    \label{eq_KGreater1_1}
    & \ \ \ \
            \| R_{r}^{ ( k_{1} ) } E \cdots E R_{r}^{ ( k_{p} ) }
               E \widehat{R}_{r}^{ (k) } 
            \|_{2} 
    \\
    & \le 
            \| R_{r}^{ ( k_{1} ) } E \cdots E R_{r}^{ ( k_{p} ) } E P_{r} \|_{\infty} 
            \| P_{r} \widehat{R}_{r}^{k} \|_{2}
          + 
            \| R_{r}^{ ( k_{1} ) } E \cdots E R_{r}^{ ( k_{p} ) } E | R_{r} |^{1/2}
            \|_{2} 
            \| | R_{r} |^{-1/2} \widehat{R}_{r}^{k} \|_{\infty}.
    \notag
  \end{align}
  Lemma \ref{lem_ElementaryNormEstimates}(i) immediately gives
  \begin{align}
    \label{eq_KGreater1_2}
        \| R_{r}^{ ( k_{1} ) } E \cdots E R_{r}^{ ( k_{p} ) } E P_{r} \|_{\infty} 
    \le g_{r}^{k} \delta_{r}'^{p} 
  \end{align}
  and since \( 1 \le k = p - k_{1} - \dots - k_{p} \) implies that there exists some \( k_{a} = 0 \), \( 1 \le a \le p \), Lemma \ref{lem_ElementaryNormEstimates}(ii) gives 
  \begin{align}
    \label{eq_KGreater1_3}
        \| R_{r}^{ ( k_{1} ) } E \cdots E R_{r}^{ ( k_{p} ) } E | R_{r} |^{1/2}
        \|_{2} 
    \le 
        g_{r}^{k - 1} \| | R_{r} |^{1/2} E P_{r} \|_{2} \delta_{r}'^{p - 1}.
  \end{align}
  The two remaining terms can directly be treated by Lemma \ref{lem_Contraction}
  (ii) and (iii).
  Together we obtain that for \( k \ge 1 \), 
  \begin{align}
    &
    \| R_{r}^{ ( k_{1} ) } E \cdots E R_{r}^{ ( k_{p} ) }
       E \widehat{R}_{r}^{ (k) } 
    \|_{2} 
    \\
    & \le 
      g_{r}^{-1/2} 
      \| | R_{r} |^{1/2} E P_{r} \|_{2} 
      \frac{ \delta_{r}'^{p} ( 1 - \delta_{r} )^{-k} }{ 1 - 2 \delta_{r} } 
      + 
      g_{r}^{-1/2} 
      \| | R_{r} |^{1/2} E P_{r} \|_{2} 
      \delta_{r}'^{p - 1} 
      ( 1 - \delta_{r} )^{ - (k + 1) } 
    \notag 
    \\
    & \le 
    g_{r}^{-1/2} 
    \| | R_{r} |^{1/2} E P_{r} \|_{2} 
    \frac{ \delta_{r}'^{p-1} ( 1 - \delta_{r} )^{-k} }{ 1 - 2 \delta_{r} }. 
    \notag
  \end{align}
  From there, it follows that
  \begin{align}
    \label{eq_KGreater1_4}
    & \ \ \ \
    \sum_{ k \ge 1 } 
    \sum_{ \substack{ k_{1}, \dots, k_{p} \ge 0 \\
                      k_{1} + \dots + k_{p} = p - k 
                    }
         } 
    \| R_{r}^{ ( k_{1} ) } E \cdots E R_{r}^{ ( k_{p} ) }
       E \widehat{R}_{r}^{ (k) } 
    \|_{2} 
    \\
    & \le 
    g_{r}^{-1/2} 
    \| | R_{r} |^{1/2} E P_{r} \|_{2} 
    \frac{ \delta_{r}'^{ p - 1 } }{ 1 - 2 \delta_{r} } 
    \sum_{ k = 1 }^{ \infty }
    2^{ 2 p - 1 - k } ( 1 - \delta_{r} )^{ - k } 
    \le 
    2g_{r}^{-1/2} 
    \| | R_{r} |^{1/2} E P_{r} \|_{2} 
    \frac{ ( 4 \delta_{r}' )^{ p - 1 } }{ 1 - 2 \delta_{r} } 
    \notag
  \end{align}
  by bounding the number of summands again via Lemma \ref{lem_NumberOfTuples}.
  Combining Equation \eqref{eq_KLess0_4} and \eqref{eq_KGreater1_4} finishes the
  proof. 
\end{proof}

Based on the projector series, we can now derive the series expansion for the eigenvalues and the corresponding perturbation result.

\begin{proof} 
  [\normalfont \textbf{Proof of Theorem \ref{thm_EigenvaluePerturbationSeries}} (Eigenvalue perturbation series)]
  \label{prf_EigenvaluePerturbationSeries} 
  If \( \delta_{r}' < 1/4 \), then Proposition \ref{thm_ProjectorPerturbationSeries} guarantees an absolutely converging perturbation series for spectral projectors.
  This yields the following decomposition
  \begin{align}
                          (\widehat{\Sigma} - \mu_{r}) \widehat{P}_{r}  
    & = 
          \widehat{P}_{r} (\widehat{\Sigma} - \mu_{r}) \widehat{P}_{r}  
      = 
          \widehat{P}_{r} E \widehat{P}_{r} + \widehat{P}_{r} ( \Sigma - \mu_{r} ) \widehat{P}_{r}
    \\
    & =
          \sum_{n = 0}^{\infty} \sum_{l = 0}^{n} P_{r}^{(l)} E                P_{r}^{(n - l)} 
        + \sum_{n = 0}^{\infty} \sum_{l = 0}^{n} P_{r}^{(l)} (\Sigma - \mu_r) P_{r}^{(n - l)} 
    \notag
    \\
    & = 
          \sum_{n = 1}^{p}          Q_{r}^{(n)} 
        + \sum_{n = p    }^{\infty} \sum_{l = 0}^{n} P_{r}^{(l)} E                P_{r}^{(n - l)}
        + \sum_{n = p + 1}^{\infty} \sum_{l = 0}^{n} P_{r}^{(l)} (\Sigma - \mu_r) P_{r}^{(n - l)}.
    \notag
  \end{align}
  Hence,
  \begin{align}
    & \ \ \ \
          \| (\widehat{\Sigma} - \mu_r) \widehat{P}_{r} - \sum_{n = 0}^{p} Q_{r}^{(n)} \|_{2}
    \\
    & \le 
          \sum_{n = p}^{\infty} \sum_{l=0}^{n} 
          \sum_{ \substack{k_{1}, \dots, k_{l + 1} \ge 0 \\
                           k_{1} + \dots + k_{l + 1} = l
                          }
               }
          \sum_{ \substack{k_{l + 2}, \dots,  k_{n + 2} \ge 0 \\
                           k_{l +2} + \dots + k_{n + 2} = n - l 
                          }
               } 
          \| R_{r}^{(k_{1})} E \cdots E R_{r}^{(k_{n + 2})} \|_{2}
    \notag
    \\
    & +   \sum_{n = p + 1}^{\infty} \sum_{l = 1}^{n} 
          \sum_{ \substack{k_{1},  \dots,  k_{l + 1} \ge 0 \\
                           k_{1} + \dots + k_{l + 1} = l \\ 
                           k_{l + 1} \ge 1
                          }
               }
          \sum_{ \substack{k_{l + 2},  \dots,  k_{n + 2} \ge 0 \\
                           k_{l + 2} + \dots + k_{n + 2} = n - l \\
                           k_{l + 2} \ge 1
                          }
               } 
         \| R_{r}^{(k_{1})} E \cdots E R_{r}^{(k_{l + 1} + k_{l + 2} - 1)} 
                            E \cdots E R_{r}^{(k_{n + 2})} \|_{2}.
    \notag             
  \end{align}
  Inserting the bounds from Lemmas \ref{lem_ElementaryNormEstimates} and \ref{lem_NumberOfTuples} to
  estimate and enumerate the summands, we arrive at
  \begin{align}
    & \ \ \ \
          \| (\widehat{\Sigma} - \mu_{r}) \widehat{P}_{r} - \sum_{n = 0}^{p - 1} Q_{r}^{(n)}\|_{2}
    \\
    & \le 4 \max \big( \| |R_{r}|^{1/2} E P_r \|_{2}, g_r^{-1/2} \| P_r E P_r \|_{2} \big) \| 
            |R_{r}|^{1/2} E P_r \|_{\infty}
            \sum_{n = p}^{\infty} (n + 1) (4 \delta_{r}')^{n - 1}
    \notag 
    \\
    & +   4 \| |R_{r}|^{1/2} E P_r \|_{2} \| |R_{r}|^{1/2} E P_r \|_{\infty}
            \sum_{n = p + 1}^{\infty} n (4 \delta_{r}')^{n - 2}.
    \notag 
  \end{align}
  Inserting Lemma \ref{lem_DerivativeOfAGeometricSeries} and the inequality $ \| |R_{r}|^{1/2} E P_r \|_{\infty}\leq g_r^{1/2}\delta_r^\prime$ yields the claim.
\end{proof}



\section{Addtional proofs of auxiliary results}
\label{sec_AddtionalProofsOfAuxiliaryResults}

\begin{lemma}[\( L_{k} \) is positive]
  \label{lem_LKIsPositive}
  Under Assumptions \ref{ass_data}, \ref{ass_Kernel} and \ref{ass_Mercer}, the kernel integral operator \( L_{k} \) is positive, i.e., for any \( f \in L^{2}(\mathbb{P}^{X}) \), we have that \( \langle f, L_{k} f \rangle \ge 0 \).
\end{lemma}

\begin{proof}[Proof]
  For any \( f \in L^{2}(\mathbb{P}^{X}) \), there exists a sequence of simple functions \( f_{n} = \sum_{i = 1}^{n} a_{i} \mathbf{1}_{A_{i}} \) such that \( f_{n} \to f \) in \( L^{2}(\mathbb{P}^{X}) \).
  For any \( \varepsilon > 0 \) and \( n \) sufficiently large, we then have
  \begin{align}
    \langle f, L_{k} f \rangle 
    & \ge 
    \langle f_{n}, L_{k} f_{n} \rangle - \varepsilon 
    = 
    \sum_{i, j = 1}^{n}
    a_{i} a_{j}
    \int_{A_{i}} \int_{A_{j}} k(x, y) \, \mathbb{P}^{X}(d x) \mathbb{P}^{X}(d y) 
    -
    \varepsilon.
    \\
    & = 
    \sum_{i, j = 1}^{n} a_{i} a_{j} \Big\langle \int_{\mathcal{X}} \mathbf{1}_{A_{i}}(x) k(x,\cdot) \mathbb{P}^{X}(d x),
            \int_{\mathcal{X}} \mathbf{1}_{A_{j}}(x) k(x,\cdot) \mathbb{P}^{X}(d x)  
    \Big\rangle_{\mathcal{H}}  - \varepsilon\\
    &=\Big\|\sum_{i}^{n} a_{i}\int_{\mathcal{X}} \mathbf{1}_{A_{j}}(x) k(x,\cdot) \mathbb{P}^{X}(d x)\Big\|_{\mathcal{H}}^2  - \varepsilon
    \ge 
    - \varepsilon, 
    \notag
  \end{align}
  Since $\varepsilon$ was arbitrary, the claim follows. Note that by the Riesz representation theorem, the integrals on the right-hand side are well defined as objects in \( \mathcal{H} \) due to the fact that \( \int_{\mathcal{X}} \| k(x,\cdot) \|_{\mathcal{H}} \mathbb{P}^{X}(d x) < \infty \), which follows from Mercer's condition.
\end{proof}

\begin{proof}[\normalfont \textbf{Proof of Theorem \ref{thm_BoundsForEigenvaluesAndEigenspaces}} (Bounds for eigenvalues and eigenspaces)]
  \label{prf_BoundsForEigenvaluesAndEigenspaces}
  From Proposition \ref{prp_NonAsymptoticBoundOnDelta_r}, it follows that \( \delta_{r}' < 1/4 \) with probability at least \( 1 - 4 \mathbf{r}_{r}(\Sigma) n^{-\gamma} \).
  On this event, Corollary \ref{cor_EigenvaluePerturbationBound} yields the bound
  \begin{align}
          \|(\widehat\lambda_j-\mu_r)_{j\in I_r}\|_{2}
    \leq
          C_{\varepsilon} ( \| P_{r} E P_{r} \|_{2} + \| |R_{r}|^{1/2} E P_{r} \|_{2} 
                                                      \| |R_{r}|^{1/2} E P_{r} \|_{\infty})
  \end{align}
  The eigenvalue statement now follows noting that Proposition \ref{prp_BoundsOnTheLinearAndQuadraticTerm} implies that for any \(\gamma > 0 \), 
  \begin{align}
    \| P_{r} E P_{r} \|_{2}^2 & \le C L^2\gamma m_r^2 \mu_r^2 \frac{\log n}{n},
    \\
    \| P_{r} E | R_{r} |^{1/2} \|_{2}^2 & \le  C L^2\gamma m_r \mu_r \sum_{s\neq r}\frac{m_s\mu_s}{|\mu_r-\mu_s|}\frac{\log n}{n} 
    \notag
  \end{align}
  and  \( \delta_{r}^\prime \leq \delta_{r} \le \sqrt{\gamma L \mu_{r} / g_{r} \mathbf{ r }_{r}(\Sigma) \frac{\log n}{n}} \) on an event with probability at least \( 1 - n^{-\gamma} \) by Proposition \ref{prp_NonAsymptoticBoundOnDelta_r}.

  The eigenfunction statement follows analogously combining Proposition \ref{prp_DeterministicBounds}, Corollary \ref{cor_ProjectorPerturbationBound} and the same probabilistic estimates.
\end{proof}

\begin{proof}[\normalfont \textbf{Proof of Proposition \ref{prp_DeterministicBounds}} (Deterministic bounds)]
  \label{prf_DeterminisicBounds}
  From Lemma \ref{lem_L_Sigma_empirical} it follows that for any \( O \in O(|I_{r}|) \), 
  \begin{align} \label{eq_EigenfunctionComparison_1}
    & \ \ \ \
          \|  \big( \widehat{v}_{j}           \big)_{j \in I_{r}} 
            - \big( S_{n} \phi_{j} / \sqrt{n} \big)_{j \in I_{r}} O 
          \|_{2}
          \\
    &  =
          \| \big( \widehat{v}_{j}              \big)_{j \in I_{r}}
           - \big( (n \mu_r)^{-1/2} S_{n} u_{j} \big)_{j \in I_{r}} O
          \|_{2}
    \notag\\
   & \le
          \Big\|   \big( \widehat{v}_{j} \big)_{j \in I_{r}} 
                 - \Big( \Big( \frac{\widehat{\lambda}_j}{\mu_r} \Big)^{1/2} \widehat{v}_{j} 
                   \Big)_{j \in I_{r}} 
          \Big\|_{2}
        +
          \Big\|   \Big( \Big( \frac{\widehat{\lambda}_j}{\mu_r}\Big)^{1/2}\widehat{v}_{j} 
                   \Big)_{j \in I_{r}}  
                -  \big( (n\mu_r)^{-1/2}S_{n} u_{j} \big)_{j \in I_{r}} O
          \Big\|_{2}
    \notag
    \\
    & = 
          \sqrt{\sum_{j \in I_r} \Big( 1 - \Big(\frac{\widehat{\lambda}_j}{\mu_r} \Big)^{1/2} \Big)^2}
        +
          \big\| (n \mu_r)^{-1/2} 
                 \big( (S_{n} \widehat{u}_j)_{j \in I_{r}} - (S_{n} u_{j})_{j \in I_{r}} O \big) 
          \big\|_{2}
    \notag
    \\
    & \le 
           \sqrt{\sum_{j \in I_{}} \Big(1- \frac{\widehat{\lambda}_{j}}{\mu_{r}}  \Big)^{2} }
          + 
          \mu_{r}^{-1/2}
          \| n^{-1/2} \big( (S_{n} \widehat{u}_j)_{j \in I_{r}} - (S_{n} u_{j})_{j \in I_{r}} O \big) \|_{2}
    \notag
  \end{align}
  where for last inequality, we have used that \( (1 - x) = (1 - \sqrt{x}) (1 + \sqrt{x}) \), \( x > 0 \).

  We abbreviate \( (u_{j})_{j \in I_{r}} \) by the operator $U_r: \mathbb{R}^{I_{r}} \to \mathcal{H} $, \( a \mapsto \sum_{j \in I_r} a_{j} u_{j} \) and \( (\widehat{u}_{j})_{j \in I_{r}} \) analogously by $ \widehat U_r $. 
  For the second term, we can the further estimate
  \begin{align}
    & 
    \| n^{-1/2} \big( (S_{n} \widehat{u}_j)_{j \in I_{r}} - (S_{n} u_{j})_{j \in I_{r}} O \big) \|_{2}^2
    \\
    &= 
    \| n^{-1/2} S_{n} (\widehat{U} - U O) \|_{2}^2
    \notag
    \\
    & = 
    n^{-1} \tr \big( \widehat U_r - U_r O \big)^{*} S_{n}^{*} S_{n} \big( \widehat U_r - U_r O \big)
    \notag
    \\
    &= 
           \tr \big( \widehat U_r - U_r O \big)^{*} \widehat{\Sigma} \big( \widehat U_r - U_r O \big)
    \notag
    \\
    & = 
           \tr \big( \widehat U_r - U_r O \big)^{*} \Sigma \big( \widehat U_r - U_r O \big)
    +
           \tr \big( \widehat U_r - U_r O \big)^{*} E \big( \widehat U_r - U_r O \big)
    \notag
    \end{align}
    Note that here, \( S_{n}^{*}: \mathbb{R}^{n}\rightarrow \mathcal{H},  w \mapsto \sum_{i = 1}^{n} w_{i} k(X_{i}, \cdot) \) is the adjoint of \( S_{n} \) as a mapping from \( (\mathbb{R}^{n}, \langle \cdot, \cdot \rangle_{\text{eu}}) \) to \( (\mathcal{H}, \langle \cdot, \cdot \rangle_{\mathcal{H}}) \).

    On the one hand, inserting into the first term that with respect to the Loewner order,
    \begin{align}
      \Sigma = \mu_rI+\sum_{s\neq r}(\mu_s-\mu_r)P_r\preceq \mu_rI+|R_r|^{-1},
    \end{align}
    we obtain that
    \begin{align}
      & \ \ \ \
            \tr (\widehat U_r - U_r O)^{*} \Sigma (\widehat U_r - U_r O)
      \\
      &\le 
            \mu_r\| \widehat U_r- U_r O \|_{2}^{2} + \tr \widehat U_r^{*} |R_r|^{-1} \widehat U_r
      \notag
      \\
      & =
            \mu_r\| \widehat U_r- U_r O \|_{2}^{2}+\tr \big(|R_r|^{-1} \widehat P_r \big)
      \notag
      \\
      &=
            \mu_r\| \widehat U_r- U_r O \|_{2}^{2}+\tr \big(\widehat P_r|R_r|^{-1} \widehat P_r  \big)
      \notag
      \\
      & \le 
            \mu_r\| \widehat U_r- U_r O \|_{2}^{2}+\||R_r|^{-1/2}\widehat P_r\|_2^2.
      \notag
    \end{align}
    By taking the infimum over \( O \in O(| I_{r} |) \), we arrive at
    \begin{align} \label{eq_EigenfunctionComparison_2}
           \tr (\widehat U_r - U_r O)^{*} \Sigma (\widehat U_r - U_r O) 
      \le 
           \mu_r \| \widehat{P}_{r} - P_{r} \|_{2}^{2} + \||R_r|^{-1/2}\widehat P_r\|_2^2.
    \end{align}
    On the other hand, inserting $I=P_r+(I-P_r)$ into the second term, we obtain that
    \begin{align}
      & \ \ \ \
      \tr (\widehat U_r - U_r O)^{*} E (\widehat U_r - U_r O)
      \\
      & =
      \tr (\widehat U_r - U_r O)^{*} P_r E P_r (\widehat U_r - U_r O)
      +
      2\tr (\widehat U_r - U_r O)^{*} P_r E (I - P_r) \widehat U_r
      \notag 
      \\
      \notag
      & \qquad \qquad \qquad \qquad \qquad \qquad \qquad \quad \ \  +
      \tr \widehat U_r^{*} (I - P_r) E (I - P_r) \widehat U_r
      \notag 
      \\
      & = 
      \tr (\widehat U_r - U_r O)^{*} P_r E P_r (\widehat U_r - U_r O )
      +
      2 \tr (\widehat U_r - U_r O)^{*} P_r E |R_r|^{1/2} |R_r|^{-1/2} \widehat P_r \widehat U_r
      \notag 
      \\
      & \qquad \qquad \qquad \qquad \qquad \qquad \qquad \quad \ \  +
      \tr \widehat P_r |R_r|^{-1/2} |R_r|^{1/2} E |R_r|^{1/2} |R_r|^{-1/2} \widehat P_r
      \notag 
      \\
      & \le 
      \| P_r E P_r \|_{\infty} \| \widehat U_r- U_r O \|_{2}^{2}
      +
      2\|\widehat U_r - U_r O \|_2 \| P_r E |R_r|^{1/2} \|_{\infty} \| |R_r|^{-1/2} \widehat P_r \|_2 \|\widehat U_r\|_\infty
      \notag 
      \\
      & \qquad \qquad \qquad \qquad \qquad \qquad +
      \||R_r|^{1/2}E|R_r|^{1/2}\|_{\infty}\||R_r|^{-1/2}\widehat P_r\|_2^2.
      \notag 
    \end{align}
    By taking the infimum over \( O \in O(| I_{r} |) \), we arrive at
    \begin{align}
      &
      \tr (\widehat U_r - U_r O \big)^{*} E (\widehat U_r - U_r O)
      \\
      & \le
      \| P_r E P_r \|_{\infty} \| \widehat{P}_{r} - P_{r} \|_{2}^{2} 
      +
      2\| \widehat{P}_{r} - P_{r} \|_{2} \|P_r E |R_r|^{1/2} \|_{\infty} \||R_r|^{-1/2}\widehat P_r\|_2
      \notag 
      \\
      & \qquad \qquad \qquad \qquad \qquad \ \ \ +
      \| |R_r|^{1/2} E |R_r|^{1/2} \|_{\infty} \| |R_r|^{-1/2} \widehat P_r \|_2^2.
      \notag 
    \end{align}
    Combining the above with the inequality 
    \begin{align}
     \label{eq:bound:somewhere:else:?}
         \|\widehat{P}_{r} - P_{r}\|_2 
     =
         \sqrt{2} \| (I - P_r) \widehat P_r \|_2 
    \le 
         \sqrt{2} g_r^{-1/2} \| |R_r|^{-1/2} \widehat P_r \|_2,
    \end{align}
    we conclude that
    \begin{align}  \label{eq_EigenfunctionComparison_3}
      &
      \tr (\widehat U_r- U_r O)^{*} E (\widehat U_r - U_r O)
      \\
      & \le 
      \big( 2 g_r^{-1} \| P_r E P_r \|_{\infty} + 2 \sqrt{2} g_r^{-1/2} \| P_r E |R_r|^{1/2} \|_{\infty}
                                                + \| |R_r|^{1/2} E |R_r|^{1/2} \|_{\infty} 
      \big)
      \||R_r|^{-1/2}\widehat P_r\|_2^2
      \notag 
      \\
      & \le 
      (2 + 2 \sqrt{2} + 1) \delta_r^\prime \| |R_r|^{-1/2} \widehat P_r \|_2^2 
      \le 
      6 \| |R_r|^{-1/2} \widehat P_r \|_2^2.
      \notag
    \end{align}
    Collecting \eqref{eq_EigenfunctionComparison_1}--\eqref{eq_EigenfunctionComparison_3} and inserting Lemma \ref{lem_Contraction}(a) and Corollary \ref{cor_ProjectorPerturbationBound}, the claim follows.
\end{proof}

\begin{proof}[\normalfont \textbf{Proof of Lemma \ref{lem_RelativeRankForTheHeatKernelInDifferentRegimes}} (Relative rank for the heat kernel)]
  \label{prf_RelativeRankForTheHeatKernelInDifferentRegimes}
  In the proof, we will use the following simple inequalities for the exponential function
  \begin{align}
    \label{eq_ElementaryExponentialInequalities}
    \frac{1}{e^x-1}      \le \frac{1}{x},
    \qquad 
    \frac{1}{1 - e^{-x}} \le  1 + \frac{1}{x},
    \quad \text{and} \quad 
    \frac{1}{e^x - 1}    \le \Big( 1 + \frac{1}{x} \Big) e^{-x}
  \end{align}
  which are valid for all $x \ge 0$ and all follow from the elementary estimate \( 1 + x \le e^{x} \).

  For (i), we only consider the case $ r \ge 2 $, the case $ r = 1 $ is simpler.
  By shifting the multiplicities \( m_{r} = 2 \) into the constnat and the radius \( R = 1 / (2 \pi) \) into the the time variable \( t \), it suffices to bound
  \begin{align}
    & 
    \sum_{s = 1}^{r - 1} \frac{e^{-t (s-1)^2}}{e^{-t (s-1)^2} - e^{-t (r-1)^2}}
    +
    \sum_{s = r + 1}^{\infty} \frac{e^{-t (s-1)^2}}{e^{-t (r-1)^2} - e^{-t(s-1)^2}}
    \\
    & +
    \frac{e^{-t (r-1)^2}}{\min(e^{-t (r-2)^2} - e^{-t(r-1)^2},e^{-t (r-1)^2} - e^{-t r^2})}.
    \notag 
  \end{align}
  We set \( \tilde s: = s - 1 \) and \( \tilde r: = r - 1 \).
  For the first term in the sum, up to a constant, we then obtain the bound
  \begin{align}
    \sum_{s = 1}^{r - 1} \frac{e^{-t \tilde s^2}}{e^{-t \tilde s^2} - e^{-t \tilde r^2}}
    & = 
    \sum_{s = 1}^{r - 1} \frac{1}{1 - e^{-t (\tilde r^2 - \tilde s^{2} )}}
    \le 
    \sum_{s = 1}^{r - 1} \frac{1}{1 - e^{-t (r (r - s))}}
    \\
    & \le 
    \sum_{s = 1}^{r - 1} \Big( 1 + \frac{1}{t r (r - s)} \Big) 
    \le 
    r + \frac{\log(e r)}{t r}.
  \end{align}
  For the second term of the sum, we similarly estimate
  \begin{align}
    \sum_{s = r + 1}^{2 r} \frac{1}{e^{t (\tilde s^{2} - \tilde r^{2})} - 1} 
    \le 
    \sum_{s = r + 1}^{2 r} \frac{1}{t (\tilde s^{2} - \tilde r^{2})} 
    \le 
    \frac{1}{t r} 
    \sum_{s = r + 1}^{2 r} \frac{1}{s - r} 
    \le 
    \frac{\log(e r)}{t r}
  \end{align}
  and 
  \begin{align}
  \sum_{s = 2 r + 1}^{\infty} \frac{e^{- t \tilde s^{2}}}{e^{- t \tilde r^{2}} - e^{- t \tilde s^{2}}} \le C\frac{1}{t r}.
  \end{align}
  Finally, with analogous arguments,
  \begin{align}\label{eq:iv:1:case}
    \frac{e^{- t (r - 1)^{2}}}{e^{- t (r - 2)^{2}} - e^{- t (r - 1)^{2}}} \le \frac{C}{t r}
    \qquad \text{ and } \qquad 
    \frac{e^{- t (r - 1)^{2}}}{e^{- t (r - 1)^{2}} - e^{- t r^{2}}} \le C \Big( 1 + \frac{1}{t r} \Big) .
  \end{align}

  For (ii) and (iii), again, we focus on the case $ r \ge 2 $.
  Throughout the proof $ c, C > 0 $ are constants depending only on $d$ that may change from line to line.
  We have
  \begin{align}
        \mathbf{r}_r(\Sigma) 
    & = \sum_{s < r} \frac{m_s e^{-t \nu_s}}{e^{-t \nu_s} - e^{-t\nu_r}} 
      + \frac{m_r e^{-t \nu_r}}{\min(e^{-t \nu_r} - e^{-t \nu_{r + 1}}, e^{-t \nu_{r - 1}} - e^{-t \nu_r})}
    \\
    & + \sum_{s > r} \frac{m_s e^{-t \nu_s}}{e^{-t \nu_r} - e^{-t \nu_s}}
    \notag
  \end{align}
  Now, the formulas for $ \nu_s $ and $ m_s $ imply that \( m_{s} \le C s^{d - 1} \) and for $s < r$, 
  \begin{align}
    \frac{e^{-t \nu_s}}{e^{-t \nu_s} -e^{-t \nu_r}}
    =
    \frac{1}{1-e^{-t(\nu_r-\nu_s)}} 
    \le
    1 + \frac{1}{t(\nu_r-\nu_s)}
    \le
    C \Big( 1 + \frac{1}{t r (r - s)} \Big).
  \end{align}
  Hence, we get
  \begin{align}
    & \ \ \ \
    \sum_{s < r} \frac{m_s e^{-t \nu_s}}{e^{-t \nu_s} - e^{-t \nu_r}}
    \le
    C \sum_{s < r} \Big( s^{d-1} + \frac{s^{d-1}} {t r (r - s)} \Big)
    \\
    & \le
    C\Big(r^d+\frac{r^{d-2}}{t}\sum_{s<r}\frac{1}{r-s}\Big)
    \le 
    C \Big( r^d + \frac{r^{d - 2} \log(er)}{t} \Big).
    \notag
  \end{align}
  Similarly, we have
  \begin{align}
    \sum_{s = r + 1}^{2 r} \frac{m_s e^{-t \nu_s}}{e^{-t \nu_r} - e^{-t \nu_s}}
    &=
    \sum_{s = r + 1}^{2 r} \frac{m_s}{e^{t (\nu_s - \nu_r)} - 1}
    \le 
    \sum_{s = r + 1}^{2 r} \frac{m_s}{t (\nu_s - \nu_r)}
    \\
    & \le 
    C \sum_{s = r + 1}^{2 r} \frac{s^{d - 1}}{t r (s - r)}
    \le 
    C \frac{r^{d - 2} \log(e r)}{t}.
    \notag 
  \end{align}
  and
  \begin{align}\label{eq:iv:2:case}
    \frac{m_r e^{-t \nu_r}}{\min(e^{-t \nu_r} - e^{-t \nu_{r + 1}}, e^{-t \nu_{r - 1}} - e^{-t \nu_r})} 
    \le 
    C \Big( 1 + \frac{1}{t r} \Big).
  \end{align}
  It remains to bound the sum over $s > 2 r$.
  First note that for $s > 2 r$, the formula for $\nu_s$ implies that \( \nu_s - \nu_r \ge c s^2 \) and
  \begin{align}
    \frac{e^{-t \nu_s}}{e^{-t \nu_r} - e^{-t \nu_s}} 
    =
    \frac{1}{e^{t (\nu_s - \nu_r)} - 1} 
    \le
    C \Big( 1 + \frac{1}{t(\nu_s - \nu_r)} \Big) e^{-t (\nu_s - \nu_r)}
    \le
    C \Big( 1 + \frac{1}{t s^2} \Big) e^{-c t s^2}.
  \end{align}
  Hence, we get
  \begin{align}
    & \ \ \ \
        \sum_{s > 2 r} \frac{m_{s} e^{-t \nu_s}}{e^{-t \nu_r} - e^{-t \nu_s}} 
    \le 
        C \sum_{s > 2 r} \Big( s^{d - 1} e^{-c t s^2} + \frac{s^{d - 3}}{t} e^{-c t s^2} \Big)
    \\
    & \le
        C \int_{2 r}^\infty \Big( x^{d - 1} e^{-c t x^2} + \frac{x^{d - 3}}{t} e^{-c t x^2} \Big) \, dx
    \le
        C \begin{cases}
            t^{-1}     \log_+( t^{-1}), & d = 2, \\ 
            t^{- d/2},              & d \ge 3. 
          \end{cases}
    \notag
  \end{align}
  Combining all of the estimates yields the statements in (ii) and (iii). 

  Finally, for (iv), by elementary calculus, we have $\log(1 - x) \ge -2 x$ for all $x \in [0, 1/2]$, implying that 
  \begin{align}
    \label{eq_TaylorLog}
    \frac{1}{1 - x} \le \exp(2 x) \qquad \text{ for all } x \in [0, 1/2].
  \end{align}
  Moreover, by the negative binomial series, we have 
  \begin{align}
    \frac{1}{(1 - x)^{d + 1}} = \sum_{r = 1}^\infty \binom{d + r - 1}{r - 1} x^{r - 1}
    \qquad \text{for all } x \in [0, 1).
  \end{align}
  Now, we compute
  \begin{align}
        \mathbf{r}_2(\Sigma) 
    & = \frac{1}{1 - e^{-t d / R^2}} 
      + \frac{(d + 1) e^{-t d / R^2}}
             {\min(e^{-t d / R^2} - e^{-2 t (d + 1) / R^2}, 1 - e^{-t d / R^2})}
      + \sum_{s > 2} \frac{m_s e^{-t \nu_s}}{e^{-t d / R^2} - e^{-t \nu_s}} 
    \notag 
    \\
    & \le     1 + \frac{R^2}{t d} 
        + C d \Big( 1 + \frac{R^2}{t d} \Big)
        + C \sum_{s > 2} m_s \Big( 1 + \frac{R^2}{t d (s - 2)} \Big) e^{-(s - 2) \frac{d t}{R^2}} 
    \\
    & \le C \Big( 1 + \frac{R^2}{t d} \Big) 
            \Big( 1 + d + \sum_{s > 2} \binom{d + s - 1}{s - 1} e^{-(s - 2) \frac{d t}{R^2}} \Big)
    \notag 
    \\
    & \le C \Big( 1 + \frac{R^2}{t d} \Big)
            \Big(1 + d + (d + 1) \sum_{s > 2} 
                                 \binom{d + s - 2}{s - 2} e^{-(s - 2) \frac{d t}{R^2}} 
            \Big)
    \notag 
    \\
    & = C \Big( 1 + \frac{R^2}{t d} \Big) 
          \Big( (d + 1) \sum_{s = 2} \binom{d + s - 2}{s - 2} e^{-(s - 2) \frac{d t}{R^2}} \Big)
    \notag 
    \\
    & = C \Big( 1 + \frac{R^2}{t d} \Big)
          \Big( \frac{d + 1}{(1 - e^{-\frac{t d}{R^2}})^{d + 1}} \Big).
    \notag 
  \end{align}
  Applying the assumption $t d / R^2 \ge \log 2$ and the inequality \eqref{eq_TaylorLog}, we conclude that
  \begin{align}
    \mathbf{r}_2(\Sigma) & \le C \Big( 1 + \frac{R^2}{t d} \Big) 
    \Big(\frac{d+1}{(1-e^{-\frac{td}{R^2}})^{d+1}}\Big)
    \leq
    C (d + 1) \exp\Big( 2 (d + 1) e^{-\frac{d t}{R^2}} \Big),
  \end{align}
  which finishes the proof.
\end{proof}

\begin{proof}[\normalfont \textbf{Proof of Lemma \ref{lem_RelativeRankForWaveletPriors}} (Relative rank for wavelet priors)]
  \label{prf_RelativeRankForWaveletPriors}
  For the first statement, we immediately have that for \( r = 0, 1, \dots, R_{\alpha} \), 
  \begin{align}
    \frac{\mu_{r}}{g_{r}} \le \frac{1}{(1 - 2^{-(2 \alpha - 1)}) \land (2^{(2 \alpha - 1)} - 1)}.
  \end{align}

  For the second statement, we write
  \begin{align}
    \mathbf{r}_{r}(\Sigma) 
    & = 
    \sum_{l = 0}^{r - 1} \frac{m_{l} \mu_{l}}{\mu_{l} - \mu_{r}} 
    + 
    \frac{m_{r} \mu_{r}}{g_{r}} 
    + 
    \sum_{l = r + 1}^{R_{\alpha}} \frac{m_{l} \mu_{l}}{\mu_{r} - \mu_{l}} 
    =: 
    (\text{I}) + (\text{II}) + (\text{III})
  \end{align}
  and consider the partial sums individually.
  We immediately have \( (\text{II}) \le C 2^{r} \) from before. 
  The tem \( (\text{I}) \) then satisfies
  \begin{align}
    (\text{I}) 
    & = 
    \sum_{l = 0}^{r - 1} \frac{m_{l} \mu_{l}}{\mu_{l} - \mu_{r}} 
    \le 
    C 2^{r} 
    + 
    \sum_{l = 0}^{j - 1} \frac{2^{l}}{1 - 2^{-(r - l) (2 \alpha - 1 )}} 
    \le 
    C 2^{r}.
  \end{align}
  Finally, the term \( (\text{III}) \) can be estimated via
  \begin{align}
    (\text{III}) & =   \sum_{l = r + 1}^{R_{\alpha}} 
                       \frac{2^{l}}{2^{(l - r) (2 \alpha - 1)} - 1} 
                   \le 
                       C 2^{r} \sum_{l = r + 1}^{R_{\alpha}} 2^{(2 - 2 \alpha) (l - r)} 
    \\
                 & \le 
                       C \begin{cases}
                           2^{(2 - 2 \alpha) R_{\alpha}} 2^{(2 \alpha - 1) r}, & \alpha \in (1/2, 1), \\
                           (R_{\alpha} - r) 2^{j},                             & \alpha = 1, \\
                           2^{r},                                              & \alpha > 1. \\ 
                       \end{cases}
    \notag
  \end{align}
  This finishes the proof.
\end{proof}

\begin{lemma}[Derivative of a geometric series]
  \label{lem_DerivativeOfAGeometricSeries}
  Let $ x \in (0, 1) $ and $ p \in \mathbb{N}_{0} $.
  Then, we have 
  \begin{align}
    \sum_{n = p}^\infty(n + 1) x^n = \frac{(p + 1) x^p}{(1 - x)^2}.
  \end{align}
\end{lemma}



\textbf{Funding.} 
Co-funded by the European Union (ERC, BigBayesUQ, project number: 101041064).
Views and opinions expressed are however those of the authors only and do not necessarily reflect those of the European Union or the European Research Council.
Neither the European Union nor the granting authority can be held responsible for them.
The research of B.S. has been partially funded by the Deutsche Forschungsgemeinschaft
(DFG) – Project-ID 318763901-SFB1294.

\printbibliography

\end{document}